\documentclass[12pt]{amsart}
\usepackage{graphicx} 
\usepackage[a4paper,margin=3cm]{geometry}
\usepackage{amsmath,verbatim}
\usepackage{enumerate}
\usepackage[titletoc]{appendix}
\usepackage[T1]{fontenc}
\usepackage[utf8]{inputenc}
\usepackage{adjustbox}
\usepackage{graphicx}
\usepackage{dsfont}
\usepackage{lineno}
\usepackage{hyperref} 
\usepackage[section]{placeins}
\hypersetup{hidelinks,
colorlinks=true,
linkcolor=black,
citecolor=black}
\usepackage{mathtools}
\usepackage{stix}
\usepackage{array}
\usepackage{amsmath,amsthm,amsfonts}
\usepackage{rotating}
\usepackage{color}
\usepackage{caption}
\newcolumntype{T}[1]{S[table-format=#1]}
\newcolumntype{U}[1]{S[table-format=#1,
                       round-mode=places, 
                       round-precision=2]}
\usepackage{marginnote}
\DeclareMathOperator*{\argmax}{arg\,max}

\newcommand{\real}{\mathbb{R}}

\newcommand{\Ee}{\mathbb E}
\newcommand{\Pp}{\mathbb P}
\newcommand{\I}{\mathbb 1}

\newcolumntype{?}{!{\vrule width 0.7pt}}

\usepackage{longtable}

\usepackage{color}

\numberwithin{equation}{section}
\theoremstyle{plain}
\newtheorem{theorem}{Theorem}
\newtheorem{lemma}[theorem]{Lemma}

\theoremstyle{definition}
\newtheorem{remark}[theorem]{Remark}
\newtheorem{definition}[theorem]{Definition}

\newtheorem{assumption}{Assumption}

\numberwithin{theorem}{section}
\usepackage{microtype}
\usepackage{soul}

\sodef\myspace{}{.2em}{1em plus1em}{2em plus.1em minus.1em}
\newcommand{\omu}[3]{\overset{\textrm{#1}}{\underset{\textrm{#3}}{#2}}}
\newcommand{\pomu}[3]{\overset{\textrm{\phantom{#1}}}{\underset{\textrm{\phantom{#3}}}{#2}}}

\newcommand{\rli}[3]{\sideset{_{\,#1}^{\mathrm{R}}}{_{#2}^{#3}}{\mathop{\mathrm{I}}}}  
\newcommand{\rlj}[3]{\sideset{_{#1}^{\mathrm{R}}}{_{#2}^{#3}}{\mathop{\mathrm{J}}}}
\newcommand{\cei}[3]{\sideset{_{~\,#1}^{\mathrm{Ce}}}{_{#2}^{#3}}{\mathop{\mathrm{I}}}}
\newcommand{\rld}[3]{\sideset{_{\,#1}^{\mathrm{R}}}{_{#2}^{#3}}{\mathop{\mathrm{D}}}}
\newcommand{\mad}[3]{\sideset{_{#1}^{\mathrm{M}}}{_{#2}^{#3}}{\mathop{\mathrm{D}}}}
 
\newcommand{\ced}[3]{\sideset{_{~\,#1}^{\mathrm{Ce}}}{_{#2}^{#3}}{\mathop{\mathrm{D}}}} 

\newcommand{\Lscr}{\mathscr{L}}

\usepackage{lipsum}
\makeatletter
\g@addto@macro{\endabstract}{\@setabstract}
\newcommand{\authorfootnotes}{\renewcommand\thefootnote{\@fnsymbol\c@footnote}}%
\makeatother
\begin{document}
\title{\bf{censored fractional Bernstein derivatives and stochastic processes}}
\maketitle
\begin{center}
\authorfootnotes
  Cailing Li
  \textsuperscript{} 
  \footnote{cailingli.math@gmail.com \\  \quad \quad This paper is part of Cailing Li's PhD thesis \cite{2023_Li} which was written at Technische Universität Dresden 
 under the supervision of Ren\'e L.Schilling. The author is grateful to David Berger for his invaluable discussion.} 

  \textsuperscript{}Institute of Mathematical Stochastics, Technische Universität Dresden, D-01217 Dresden, Germany \par \bigskip
\end{center}
\section*{Abstract}
In this paper, we define the censored fractional Bernstein derivative on the positive half line $(0, \infty)$ based on the Bernstein Riemann--Liouville fractional derivative. This derivative can be shown to be the generator of the  censored subordinator by solving a resolvent equation. We also show that the  censored subordinator hits the boundary in finite time under certain conditions.

\section{Introduction}
In this paper we focus on a special class of L\'evy processes see e.g.\ \cite{2013_Boettcher,1999_Satoa}, the so--called subordinators. A \textbf{L\'evy process} is a stochastic process with c\`adl\`ag (right--continuous, finite left limits) paths and independent and stationary increments. A \textbf{subordinator} is a L\'evy process $S$ with $S_0=0$ and a.s.\ increasing paths. In this case  the Laplace transform defines $S$ uniquely; the Laplace transform is given by,   
\begin{equation}\label{bernstein}
    \mathds E\left(e^{- \lambda S_t}\right) = e^{-t f(\lambda )}, \quad  \lambda>0,
\end{equation}
where $f$ is a \textbf{Bernstein function}. A Bernstein function can be expressed by
\begin{equation}\label{sub-bf-0}
    f(\lambda) = b\lambda + \int_0^\infty (1-e^{-\lambda x})\,\mu(dx), \quad  \lambda>0,
\end{equation}
where $b\geq 0$ is the drift and $\mu$ is the jump measure of the subordinator on $(0,\infty)$ such that $\int_0^\infty \min\{1,x\}\,\mu(dx)<\infty$.
There is a deep connection between generators of subordinators and fractional derivatives. Let us recall some facts about fractional derivatives in general.
 

%


Fractional derivatives, in particular fractional time derivatives, have recently become important tools to model real-world phenomena. There are important applications in Physics, Chemistry and Biology 
as well as in Mathematics. A good introduction to applications is given in the monograph Klages el al \cite{2008_Klages}.
 Standard references are \cite{2008_Klages, 2010_Diethelm} and \cite{1993_Samko}. 
For a general background in fractional calculus and fractional differential equations, we refer to  \cite{2010_Diethelm, 2006_Kilbas, 1998_Podlubny, 2007_Mainardi, 2018_HernandezHernandez, 2016_HernandezHernandez, 2009_Meerschaert, 2017_Chen, 2018_Sin}.
Our starting point is the so--called \textbf{Marchaud fractional derivative},  which is given by
\begin{align}
\label{mar-def-001}
    \mad{}{+}{\alpha} \phi(x)&=\frac{\alpha}{\Gamma(1-\alpha)}\int_{0+}^\infty \frac{\phi(x)-\phi(x-s)}{s^{1+\alpha}}\,ds.
\end{align}

The advantage of the form \eqref{mar-def-001} is that it relaxes on the regularity requirements on  $\phi$, and that it shows the rationale behind the derivative: it is a limit of weighted sums of the increments of $\phi$, $\phi(x)-\phi(x-s)$ from various past values (i.e. in $(-\infty, x)$)  up to the present value $t$. This allows us to  generalize fractional derivatives by using a different kind of positive weights, using the theory of Bernstein functions \cite{2012_Schilling}. Moreover this establishes in a natural way the connection to L\'evy processes. Recently, censored fractional derivatives see below are studied in \cite{2021_Du}.  The \textbf{censored fractional derivative} is given by 
\begin{equation}\label{cen-der-002}
 \ced{0}{x}{\alpha}\phi(x)=\frac{\alpha}{\Gamma(1-\alpha)}\int_0^x \left(\phi(x)-\phi(x-s)\right) s^{-\alpha-1}\,ds,
 \end{equation}
 where $\alpha\in (0, 1)$. Notice that the input extends over $(0,x)$ rather than $(0, \infty)$, due to "censoring" .

Based on the Marchaud fractional derivative, we have an alternative choice to generalize this,  starting from  a Bernstein function (or the generator of  a subordinator).  To do so, it is useful to recall the following:
%
%
%

Since $f$ from \eqref{bernstein} is given  by \eqref{sub-bf-0}, the infinitesimal generator of the subordinator $S$ is formally given by 
\begin{align*}
\mathcal{A} u(x)= \,b\cdot \frac d{dx}u(x)+\int_0^\infty \left(u(x+t)-u(x)\right)\mu(dt) 
\end{align*}
 for suitable functions $u:\, \real^+  \to \real$. We can demonstrate that the formal adjoint $\mathcal{A}^*$ in the space $L^2(0, \infty)$ can be characterized using different types of extensions, such as killing--extension ($u(x)\mathds{1}_{(-\infty,0)}(x)=0$), sticky--extension ($u(x)\mathds{1}_{(-\infty,0)}(x)=u(0)$), and even--extension ($u(x)\mathds{1}_{(-\infty,0)}(x)=u(-x)$). Each type of extension influences the properties of $\mathcal{A}^*$ in a distinct manner is given by 
 \begin{align*}
\mathcal{A}^*u(x)= \,-b\cdot \frac d{dx}u(x)-\int_0^\infty \left(u(x)-u(x-t)\right)\mu(dt).
\end{align*}
Comparing this formula with the classical Marchaud derivative \eqref{mar-def-001}, the candidate for a general Marchaud fractional derivative, which we call  \textbf{Bernstein Marchaud fractional derivative},  is 
\begin{equation}\label{mau-fra-ber}
\mad{}{+}{f}u(x)
    = \,b\cdot \frac d{dx}u(x)+\int_0^\infty \left(u(x)-u(x-t)\right)\mu(dt).
\end{equation}
 Observe that $\mad{}{+}{f}$ is connected with the Bernstein function $f$ and its unique representing pair $(b, \mu)$.  Moreover $-\mad{}{+}{f}$ is the generator of an adjoint subordinator, but subject to a boundary condition that must reflect the type of  the extension. 

 We obtain  the \textbf{Bernstein Riemann--Liouville derivative} \eqref{re-fra-ber} using the killing type extension 
from the Bernstein Marchaud fractional derivative \eqref{mau-fra-ber}.
\begin{equation}\label{re-fra-ber}
\rld{0}{x}{f}\phi= b\cdot \frac d{dx}\phi^0 (x)+\int_0^\infty \left(\phi^0(x)-\phi^0(x-s)\right) \mu(ds),
\end{equation}
where $\phi^0$ 
is the  killing--type extension by setting $\phi|_{(-\infty,0)}=0$ (\textbf{killing--type extension} $\phi^0$). 

In this paper, we  introduce a general censored fractional derivative and  study its properties. We  discuss various aspects such as the resolvent equation and the construction of the corresponding Markov process. By investigating these topics, we aim to gain a deeper understanding of the characteristics and behavior of the censored fractional derivative.
 Examining the structure of \eqref{cen-der-002},  a good candidate for a generalized censored fractional derivative is
\begin{equation}\label{cen-gen-8}
\ced{0}{x}{f}\phi(x)=\int_0^x \left(\phi(x)-\phi(x-s)\right)\mu(ds).
\end{equation}
 Our starting point is that the operator  given by \eqref{cen-gen-8}  can be viewed as generator of a decreasing subordinator $x-S_t$, where we allow  only  those jumps such that the generator lands inside $(0,\infty)$ and  we suppress the jumps such that the subordinator would land outside of $(0, \infty)$. 
 In order to show that $\ced{0}{\cdot}{f}$  generates a stochastic process we apply  the Hille--Yosida theorem \cite{2012_Pazy}, which leads to  the necessity to solve the following resolvent equation in  $C_{\bar{\mu}}[0,T]$ (see p. \pageref{spa-re-eq} for the definition)
\begin{equation}\label{res-eq-9}
\left \{\begin{aligned}
\ced{0}{x}{f}\phi(x)&=\lambda \phi(x), \quad \, x>0, \quad \lambda\in \real,
\\ \phi(0)&=\phi_0.
\end{aligned}
\right.
\end{equation}
 A key step to solve \eqref{res-eq-9}, we need so--called \textbf{Sonine pairs} in Section 2 and we refer the reader to \cite{2011_Kochubei, 2002_Samko, 2003_Samko, 2020_Hanyga, 2022_Luchko}  and p. \pageref{son-p-3-1} in this paper. 

Finally  we prove that \eqref{cen-gen-8} is the generator of the censored subordinator $S^c=(S^c_t)_{t\geq 0}$ in $(0, \infty)$ and show that, under certain assumptions, this process has the Feller property. 

\section{Basics}
In this section, we will present some basic results about the Bernstein Riemann--Liouville fractional derivative. For a general study of fractional derivatives, we refer to Samko et.al. \cite{1993_Samko}.
\begin{definition}[Extension]\label{mar-ext}
    Let $\phi : [0,\infty)\to\real$. Then we can extend $\phi$ in the following way to become a function on $\real$
    \begin{itemize}
    \item `Killing type' extension: $\phi^0(x)=\phi(x)\I_{[0,\infty)}(x) + 0 \cdot \I_{(-\infty,0)}(x)$.
    \end{itemize}
\end{definition}
We will now introduce the definitions of two key concepts: the Marchaud fractional derivative and the Riemann--Liouville fractional derivative. Additionally, we will discuss the Riemann--Liouville fractional integral, all of which are induced by a Bernstein function denoted as $f$.
\begin{definition}\label{mar-defin}
Let $\phi:\real\to\real$ be a function which is defined on the whole real axis and let $f$ be a Bernstein function given by \eqref{sub-bf-0}. The \textbf{Bernstein Marchaud fractional derivatives for the Bernstein function $f$} are the following operators
\begin{align}
\label{bf-mar-def+}
    \mad{}{+}{f}\phi(x)
    := \,b\cdot \frac d{dx}\phi(x)+\int_0^\infty \left(\phi(x)-\phi(x-t)\right)\mu(dt)
%
\end{align}
\end{definition}
\begin{definition}\label{mar-int-defin}
Let $\phi:\real\to\real$ be a function which is defined on the whole real axis and  let $f$ be a Bernstein function given by \eqref{sub-bf-0}.  The \textbf{Riemann--Liouville fractional integral for the Bernstein function $f$} are the following operators
\begin{align}
\rli{0}{x}{f}\phi(x)=\int_0^x \phi(z)k(x-z)dz
\end{align}
where $k$ is defined by the following  Laplace transfrom $\Lscr[k;\lambda]=1/f(\lambda)$.
\end{definition}
By applying the extension idea, we can obtain an alternative expression of the Bernstein Riemann--Liouville fractional derivative using the Bernstein Marchaud fractional derivative.
\begin{definition}\label{ber-mar-def}
Let $\phi$ be a function on $[0,\infty)$ and $f$ any Bernstein function with characteristics $(b,\mu)$. The 
\textbf{Bernstein Riemann Liouville fractional derivatives}, 
     are defined by
\begin{align}
    \rld{0}{x}{f}\phi
    &:=\mad{}{+}{f}\phi^0(x)= b\cdot \frac d{dx}\phi^0(x)+\int_0^\infty \left(\phi^0(x)-\phi^0(x-s)\right)\mu(ds),\label{bf-rl-def+}
\end{align}
\end{definition}
\begin{remark}
\label{ber-def-rei}
Let $\phi$ be a function on $[0,\infty)$ and $f$ be a Bernstein function with characteristics $(b,\mu)$. The Bernstein Riemann--Liouville fractional derivatives, 
are
\begin{align}
    \rld{0}{x}{f}\phi &= b \cdot \frac d{dx}\phi(x)+ \frac d{dx} \int_0^x \phi(t)\,\bar\mu(x-t)\,dt, \label{bf-fra-rl+}
\end{align}
 where $\bar\mu(s):=\mu(s, \infty)$. 
\end{remark}

 Next we introduce a special Volterra equation first kind of Sonine equation, whose solutions are known as the Sonine pairs , given by equation \eqref{son-equ}. The formal solution to such equations was discovered a long time ago and has been extensively studied, see Samko, Kilbas, and Marichev \cite{1993_Samko}. 
\begin{definition}\label{son-p-3-1}  Let $h,g \in L^1_{\mathrm{loc}}(0, \infty)$. We call $h, g$ a \textbf{positive Sonine pair} if 
\begin{equation}\label{son-equ}
h*g(x)=1, \quad \forall x>0.
\end{equation}
where $h, g \geq 0$ are  either both functions or, by abuse of notation, one of them is  a positive function on $(0, \infty)$ and the  other  is a positive measure on $(0, \infty)$. 
\end{definition}
 
It is easy to see that a  generalized Riemann--Liouville fractional derivative defined by  $\rld{0}{\cdot}{f}$ (see Remark \ref{ber-def-rei})  and the associated generalized fractional integral $\rli{0}{\cdot}{f}$ (see Definition \ref{mar-int-defin}) are given by kernels $\bar\mu$ and $k$ which are a Sonine pair (see Appendix Theorem A.1)

If we consider $k(x)=\frac{x^{\alpha-1}}{\Gamma(\alpha)}$ and $\bar{\mu}(x)=\frac{x^{\alpha}}{\Gamma(1-\alpha)}$, then the $\alpha$-order fractional integral and derivative can be defined, and their convolution is equal to 1, as  in eq.  \eqref{son-equ}.

The connection between complete Bernstein functions and Sonine pairs  have been
 noticed by \cite{2011_Kochubei}, 
and there are significant applications due to the “nonlinear” properties of complete Bernstein functions, for details and examples of Sonine pairs see \cite{2023_Li}. 


There is a deep one--to--one correspondence between Sonine pairs and so--called special Bernstein functions -- but this is a highly theoretical result, since we know  not much about special Bernstein functions cf. \cite{2012_Schilling}. Therefore, we use the more benign class of complete Bernstein functions.  Thus, we will investigate the connection between complete Bernstein functions and Sonine pairs  in the appendix.
 
Throughout this work, we always make the following Assumption 1.
 \begin{assumption}\label{ass-a}
\begin{enumerate}[(1)]
\item $f$ is  a complete Bernstein function with characteristics $(0, \mu)$, i.e.  the L\'{e}vy measure $\mu$ has a completely monotone density $m$ with respect to the Lebesgue measure.
\item  
The following limiting relations are here
\begin{align}
\lim_{\lambda\to0}\frac1{f(\lambda)}= \infty, \quad \lim_{\lambda\to \infty}\frac1{f(\lambda)}\to 0 \label{s-inf-3-4};
\\\lim_{\lambda\to0}\frac{\lambda}{f(\lambda)}=0, \quad \lim_{\lambda\to \infty}\frac{\lambda}{f(\lambda)}\to\infty \label{s-inf-fun}.
\end{align}
 \end{enumerate}
 \end{assumption}
Under these assumptions, we have the following Laplace transform: 
\begin{equation}\label{bf-pot}
\frac{f(\lambda)}{\lambda}=\int_0^\infty e^{-\lambda x}\bar{\mu}(x)\,dx
\end{equation}
and 
\begin{equation}\label{pot-eq}
\frac1{f(\lambda)}=\int_0^\infty e^{-\lambda x}k(x)\,dx, 
\end{equation}
where $k$ is the potential  density and $\bar\mu(s):=\mu(s, \infty)$. 

Now let us define the following function spaces:

$$C_{\bar{\mu}}(0,T]=\left\{\varphi(x) \in C (0, T]\cap L^1(0,T]: (\bar{\mu}*\varphi)(x)\in C^1(0,T]\right\},$$ 
$$C_{\bar{\mu}}[0,T]=C[0,T]\cap C_{\bar{\mu}}(0,T].$$\label{spa-re-eq}

\noindent If we take the kernel $\bar{\mu}(x)=\frac{x^{-\beta}}{\Gamma(1-\beta)}$ the above space is the same  as  $C_\beta(0,T]$ in  Du, Toniazzi and Xu \cite{2021_Du}. 
Let us define the following Bernstein Riemann--Liouville fractional derivative for $\phi \in C_{\bar{\mu}}(0,T]$  and Bernstein Riemann--Liouville integral $\phi\in  C(0, T] \cap L^1(0,T]$,

\begin{equation}\label{brl-d}
\rld{0}{x}{f}\phi = \frac d{dx} \int_0^x \phi(s)\,\bar\mu(x-s)\,ds,
\end{equation}

\begin{equation}\label{brl-i}
\rli{0}{x}{f}\varphi=\int_0^x \phi (s)k(x-s)\,ds.
\end{equation}

\begin{theorem}\label{inv-berf-der} If $\phi\in C(0,T]\cap L^1(0,T]$, then $\rli{0}{x}{f}\phi \in C_{\bar{\mu}}(0,T]$ and $ \rld{0}{\cdot}{f}$ is the left inverse of $\rli{0}{\cdot}{f}$.
\end{theorem}
\begin{proof} For the first claim, we need to show $\rli{0}{x}{f} \phi  \in C(0, T]\cap L^1(0,T]$ and \linebreak $(\bar{\mu}*\rli{0}{\cdot}{f})(x)\in C^1(0,T]$.
Firstly we show that $\rli{0}{x}{f}\phi \in C(0, T]\cap L^1(0,T]$. Since  $\phi$ and $k$ are in $C(0, T]\cap L^1(0,T]$, $\rli{0}{x}{f}\phi$ is well--defined and finite. 
 Let $\varepsilon \in (0, x)$ 
\begin{equation}
\rlj{0}{x-\varepsilon}{f}\phi=\int_0^{x-\varepsilon} \phi(s)k(x-s)\,ds,
\end{equation}
note  that $\rlj{0}{x-\varepsilon}{f}=\rli{0}{x-\varepsilon}{f}$.  Given $T_0\in (0,T]$, for all $x\in [T_0, T]$ and $\varepsilon\in (0, T_0)$, we have 
\begin{align*}
\left| \rlj{0}{x-\varepsilon}{f}\phi-\rli{0}{x}{f}\phi \right|&=\left|\int_{x-\varepsilon}^x \phi(s)k(x-s)\,ds \right|
\\&\leq \int_{x-\varepsilon}^x \left| \phi(s) k(r-s) \right| \,ds 
\\&\leq \|\phi\|_{C[T_0-\varepsilon, T]} \int_{x-\varepsilon}^x k(r-s) \,ds,
\end{align*}
Using the local integrability of $k$, we conclude that the above limit goes to 0 uniformly, as $\varepsilon$ goes to 0. Applying the dominated convergence theorem, $\rlj{0}{x-\varepsilon}{f}\phi$ is continuous on $[T_0, T]$. So,  being a uniform limit of continuous function, $\rli{0}{x}{f}\phi$ is continuous on $[T_0, T]$, and thus on $(0, T]$.    The integrability of $\rli{0}{x}{f}\phi$ follows by
\begin{align*}
\int_0^T \left|\rli{0}{x}{f}\phi\right|\,dr &\leq \int_0^T \int_0^x \left|\phi(r)\right| k(x-r)\,dr\,dx
\\&=\int_0^T \left|\phi(r)\right| \int_r^T  k(x-r) \,dx \,dr
\\&\leq \int_0^T  k(r) \,dr \int_0^T \left|\phi(r)\right|  \,dr
<\infty.
\end{align*}
$\rli{0}{\cdot}{f}\in C(0, T]\cap L^1(0, T]$, and $\frac{d}{dx}\bar{\mu}*\rli{0}{x}{f}\phi$ is well--defined on $(0, T]$. For $x\in (0, T]$,
\begin{equation*}\label{lef-in-410}
\begin{aligned}
\bar{\mu}*\rli{0}{x}{f}\phi&\pomu{Fubini}{=}{}\int_0^x\bar{\mu}(x-r)\rli{0}{r}{f}\phi(r)\,dr
\\&\pomu{Fubini}{=}{}\int_0^x\phi(s)\int_s^x \bar{\mu}(x-r)k(r-s)\,dr\,ds
\\&\pomu{r=s+t}{=}{}\int_0^x\phi(s)\,ds
\end{aligned}
\end{equation*}
 then we get $(\bar{\mu}*\rli{0}{\cdot}{f})\phi\in C^1(0,T]$.  Using the definition of the Bernstein Riemann--Liouville fractional derivative Remark \ref{ber-def-rei} and  noting that  we assume $b=0$,  we get   $\rld{0}{x}{f}\rli{0}{x}{f}\phi=\frac{d}{dx}\left(\bar{\mu}*\rli{0}{x}{f}\phi\right)=\phi$,  if we take derivative on both sides of \eqref{lef-in-410}. 
\end{proof}
\begin{theorem} \label{uni-riem-spac}
\begin{enumerate}[(i)]
\item If $\psi \in C(0, T]\cap L^1(0, T]$. Then $\phi= \rli{0}{x}{f}\psi$ if and only if $\phi\in C_{\bar{\mu}}(0,T],$ $\rld{0}{x}{f}\phi=\psi$ and $\lim_{x\to0} (\bar{\mu}*\phi)(x)=0$.
\item If $\phi  \in C_{\bar{\mu}}[0, T]$ and $\rld{0}{x}{f}\phi=0$, then $\phi=0$.
\end{enumerate}
\end{theorem}
\begin{proof}
(i)   "$\Longrightarrow$"  Using the proof of Theorem \ref{inv-berf-der}, we have $\phi\in C_{\bar{\mu}}(0,T]$ and $\rld{0}{x}{f}\phi=\psi$. By assumptions we can easily get $\lim_{x\to0} (\bar{\mu}*\phi)(x)=0$.

 "$\Longleftarrow$" Using the results of Theorem \ref{inv-berf-der}, $ \rld{0}{\cdot}{f}$ is the  left inverse of $\rli{0}{\cdot}{f}$. Then we have $\rld{0}{x}{f}\phi=\psi=\rld{0}{x}{f} \rli{0}{x}{f}\psi$.  Consequently  $\rld{0}{x}{f}(\phi-\rli{0}{x}{f}\psi)=0$ and  $\bar{\mu}*(\phi-\rli{0}{x}{f}\psi)$ is a constant.  By \eqref{lef-in-410}, we know $\lim_{x\to0}\bar\mu*\rli{0}{x}{f}\psi=0$, and  $\lim_{x\to0} (\bar{\mu}*\phi)(x)=0$ by assumption. Therefore, $\bar{\mu}*(\phi-\rli{0}{x}{f}\psi)$  must be 0. Using the results of Theorem \ref{inv-berf-der}, $\rld{0}{\cdot}{f}$ is the left inverse of $\rli{0}{\cdot}{f}$.   We can conclude that 
\begin{align*}
\phi-\rli{0}{x}{f}\psi&\pomu{3}{=}{}\rld{0}{x}{f}\rli{0}{x}{f}[\phi-\rli{0}{x}{f}\psi]
\\& \pomu{3}{=}{}\frac{d}{dx}\left[\bar\mu*(\rli{0}{\cdot}{f}(\phi-\rli{0}{\cdot}{f}\psi))\right]
\\&\omu{(*)}{=}{}\frac{d}{dx}\left[k*(\bar\mu*(\phi-\rli{0}{\cdot}{f}\psi))\right]
\\&\pomu{3}{=}{}\frac{d}{dx}[k*0]
\\&\pomu{3}{=}{}0,
\end{align*}
In the equality  marked by (*), we use the associativity of convolution.
 Thus, We can obtain $\phi-\rli{0}{x}{f}\psi=0$.

(ii) For $\phi\in C_{\bar{\mu}}[0, T],  (\bar{\mu}*\phi)(x)\leq \|\phi\|_\infty \int_0^x\bar{\mu}(s)ds\to 0, \quad x\to 0$.
Using the results of (i), $\phi=\rli{0}{x}{f}0=0$.
\end{proof}
\section{Censored initial value problem}
 We now proceed to define the censored Bernstein fractional derivative, which is the primary focus of our study.
\begin{definition}\label{cen-bef-def}
We define the \textbf{censored Bernstein fractional derivative} of any $\varphi\in C_{\bar{\mu}}(0,T]$ as 
\begin{equation}
\ced{0}{x}{f}\varphi(x)=\rld{0}{x}{f}\varphi(x)-\varphi(x)\bar{\mu}(x), \quad x\in (0,T].
\end{equation}
\end{definition}

We introduce an integral operator, related kernels, and a key condition essential for the convergence study of \eqref{3-19}. These components are fundamental to our analysis and will be discussed in detail. 
\begin{definition}\label{ker-def-1}
For $\phi\in C[0,T]$, we define 
\begin{equation} \label{ker-def}
\mathcal{K}\phi(x)=\left \{ \begin{aligned} & \int_0^xk_1(x,r)\phi(r)\,dr, & x>0;\\
&\phi(0),  &x=0.
\end{aligned}
\right.
\end{equation}
Furthermore $\mathcal{K}\phi(x)=\rli{0}{x}{f}[\bar{\mu}\phi]$ and $\mathcal{K}$ is a linear operator preserving positivity positivity(i.e.  $\mathcal{K} \phi\geq0$ if $\phi\geq0$). In addition $\mathcal{K}^j \phi(x)=\int_0^x k_j(x,r)\phi(r)\,dr$ by induction. 
where for $0<r<x$, $k_j(x, r)$ is the following kernel
\begin{equation} \label{kernel-def}
k_j(x,r)=\left \{ \begin{aligned}&\bar{\mu}(r)k(x-r), & j=1;\\
&\int_r^xk_1(x,s)k_{j-1}(s, r)\,ds, &j\geq2.
\end{aligned}
\right.
\end{equation}
Furthermore $\int_0^x k_j(x,r)\,dr=1$ by induction.
\end{definition}

\begin{lemma}\label{con-tra} Let Assumption \ref{ass-a} hold.   Let $k,\bar\mu$ be a Sonine pair and $\Lscr(\bar\mu, \lambda)=f(\lambda)/\lambda$, $\Lscr(k, \lambda)=1/f(\lambda)$.   The limit $q=\lim\limits_{x\to 0}\bar{\mu}(x)\int_0^x k(s)\,ds$ is equal to $1$ if, and only if,  the supremum $\sup_{x\in [0,T]}\bar{\mu}(x)\int_0^x k(s)\,ds $ is equal to $1$ for some $T>0$.
\end{lemma}
\begin{proof} According to Assumption \ref{ass-a}, we have that $\bar{\mu},k \in \mathcal{CM}$ and  $\bar{\mu}, k$ are decreasing. Thus, 
\begin{align*}
\bar{\mu}(x)\int_0^x k(s)\,ds&=\bar{\mu}(x)\int_0^x k(x-s)\,ds
\\&\leq \int_0^x \bar{\mu}(s)k(x-s)\,ds
\\&=1
\end{align*}
Then we have $q \leq 1$.  Assume that for every $T>0$,  there is some some $x_0\in [0, T]$ such that $\bar{\mu}(x_0)\int_0^{x_0} k(s)\,ds=1$. Since  $\int_0^{x_0}\bar{\mu}(s) k(x_0-s)\,ds=1$, we have 
\begin{align*}
 0&=\bar{\mu}(x_0)\int_0^{x_0} k(s)\,ds-\int_0^{x_0}\bar{\mu}(s) k(x_0-s)\,ds
 \\&= \int_0^{x_0} \bar{\mu}(x_0)k(x_0-s)\,ds-\int_0^{x_0}\bar{\mu}(s) k(x_0-s)\,ds
 \\&= \int_0^{x_0} (\bar{\mu}(x_0)-\bar{\mu}(s)) k(x_0-s)\,ds.
\end{align*}
This  means that  $\bar{\mu}$ is constant  in $[0, x_0]$ since $\bar\mu$ is decreasing. Thus  for all $x\in[0, x_0]$  we see from the Sonine equation that $\int_0^{x}k(s)\,ds $  is  constant on $[0, x_0]$.  Therefore,  $k=0$ on $[0,x_0]$. This is a contradiction to $\bar{\mu}, k$ being a Sonine pair. 
\end{proof}
We are now going to establish the crucial bound and continuity properties necessary for adapting our approach to the censored initial value problem
\eqref{ivp-15}.
\begin{lemma}\label{ope-bou}
Let Assumption \ref{ass-a}  and $q=\lim\limits_{x\to 0}\bar{\mu}(x)\int_0^x k(s)\,ds<1$  hold true.  Let $k,\bar\mu$ be  a Sonine pair and $\Lscr(\bar\mu, \lambda)=f(\lambda)/\lambda, \quad \Lscr(k,\lambda)=1/f(\lambda)$.   Assume that  $\varphi \in C[0,T]$ and $ |\varphi(x)|\leq M\int_0^x k(s)\,ds$, where $M$ is a positive constant,   then for all $x\in [0,T]$, $\mathcal{K}\varphi(x)\leq M q \int_0^x k(s)\,ds$. Furthermore, $ |\mathcal{K}^j \varphi(x)|\leq M q^j \int_0^x k(s)\,ds$, $x\in [0,T]$  and $\mathcal{K}\varphi(x)\in C[0,T]$.
\end{lemma}
\begin{proof} According to Assumption \ref{ass-a}, 
we have that $\bar{\mu},k$ are completely montone functions and  $\bar{\mu}, k$ are decreasing.
For the first claim, we have \begin{align*}
|\mathcal{K}\varphi(x)|&=\left|\int_0^x \bar{\mu}(r) k(x-r) \varphi(r)\,dr\right|
\\&\leq  \sup_{r\leq x} \varphi(r) \bar{\mu}(r)\int_0^x  k(s)\,ds
\\&\leq M \int_0^x k(s)\,ds \sup_{r\leq x} \bar{\mu}(r)\int_0^r  k(s)\,ds
\\&= M q \int_0^x k(s)\,ds,
\end{align*}
\noindent where we use the condition $|\varphi(x)|\leq M\int_0^x k(s)\,ds$ and $q=\lim\limits_{x\to 0}\bar{\mu}(x)\int_0^x k(s)\,ds$  for the estimate, see  Lemma \ref{con-tra}. 
For the second claim,  we use  the first claim and iterate it,
\begin{align*}
|\mathcal{K}^j \varphi(x)|&\omu{\eqref{ker-def}}{=}{} \left|\mathcal{K}^{j-1} \mathcal{K}\int_0^x\bar{\mu}(r) k(x-r) \varphi(r)\,dr\right|
\\&\pomu{\eqref{ker-def}}{\leq}{} M q \mathcal{K}^{j-1} \int_0^x k(r) \,dr
\\&\pomu{\eqref{ker-def}}{\leq}{} M q^2 \mathcal{K}^{j-2} \int_0^x k(x-s) \,ds
\\&\pomu{\eqref{ker-def}}{\leq}{} M q^j \int_0^x k(s) \,ds. 
\end{align*}
This finishes  the  proof of the second claim.

Now we prove $\mathcal{K}\varphi$ is continuous in $[0,T]$. The first case is if $x=0$, we show ${\displaystyle \lim_{x\to 0}\mathcal{K}\varphi(x)=\varphi(0)}$.
\begin{align*}
|\mathcal{K}\varphi(x)-\varphi(0)|&\omu{\eqref{ker-def}}{=}{}\left|\int_0^x\bar{\mu}(r) k(x-r)(\varphi(r)-\varphi(0))\,dr\right|
\\&\pomu{\eqref{kernel-def}}{\leq}{} \int_0^x\bar{\mu}(r) k(x-r) \left|\varphi(r)-\varphi(0)\right|\,dr
\\&\pomu{\eqref{kernel-def}}{<}{} \varepsilon.
\end{align*}
The last inequality holds due to  the  continuity of $\varphi$ at 0.
For the second case, assume $x>0$ and $x>y$;   we prove $\mathcal{K}\varphi$ is continuous in $(0, T]$.    For all $\delta>0$, we consider $C[\delta, T]$
\begin{align*}
|\mathcal{K}\varphi(x)-\mathcal{K}\varphi(y)|&\omu{\eqref{ker-def}}{=}{}\left|\int_0^x\bar{\mu}(r) k(x-r) \varphi(r)\,dr-\int_0^y \bar{\mu}(r) k(y-r)\varphi(r)\,dr\right|
\\&\pomu{\eqref{kernel-def}}{\leq}{} \left|\int_0^y\bar{\mu}(x-r) k(r) \varphi(x-r)\,dr-\int_0^y \bar{\mu}(y-r) k(r)\varphi(y-r)\,dr\right| 
\\& \quad \pomu{\eqref{kernel-def}}{+}{}  \left|\int_y^x\bar{\mu}(x-r) k(r) \varphi(x-r)\,dr\right|
= \text{I+II}.
\end{align*} 
For the term II, due to $\varphi\in C[0, T]$,   $k$ being  finite on $[\delta, T]$ and  $\bar{\mu}$ being locally integrable,  the term  II goes to 0,  as  $x$ goes to $y$. 
For Part I,  we have the following estimate, 
\begin{align*}
 &\left|\int_0^y\bar{\mu}(x-r) k(r) \varphi(x-r)\,dr-\int_0^y \bar{\mu}(y-r) k(r)\varphi(y-r)\,dr\right|
\\&\quad+ \int_0^y \bar{\mu}(y-r) k(r)\varphi(x-r)\,dr-\int_0^y \bar{\mu}(y-r) k(r)\varphi(y-r)\,dr\bigg| 
\\&\leq  \bigg|\int_0^y\bar{\mu}(x-r) k(r) \varphi(x-r)\,dr-\int_0^y \bar{\mu}(y-r) k(r)\varphi(x-r)\,dr\bigg|
\\&\quad +\bigg|\int_0^y \bar{\mu}(y-r) k(r)\varphi(x-r)\,dr-\int_0^y \bar{\mu}(y-r) k(r)\varphi(y-r)\,dr\bigg| 
\\&\leq \int_0^y \left| \bar{\mu}(x-r) -\bar{\mu}(y-r) \right| k(r) \varphi(x-r) \,dr\\&\quad+\int_0^y \left | \varphi(x-r)-\varphi(y-r) \right | \bar{\mu}(y-r) k(r) \,dr
\\&=\text{III+IV}
\end{align*}
  Due to $\varphi\in C[0, T]$, by Lebesgue's dominated convergence theorem, we obtain that the term IV goes to 0,  as $x$ goes to $y$. 
For the term III, we have for all $\varepsilon>0$, 
\begin{align*}
&\int_0^y \left| \bar{\mu}(x-r) -\bar{\mu}(y-r) \right| k(r) \varphi(x-r) \,dr
\\&= \int_0^{y-\varepsilon} \left| \bar{\mu}(x-r) -\bar{\mu}(y-r) \right| k(r) \varphi(x-r) \,dr
\\&\quad+ \int_{y-\varepsilon}^y \left| \bar{\mu}(x-r) -\bar{\mu}(y-r) \right| k(r) \varphi(x-r) \,dr
\\& =\int_0^{y-\varepsilon} \left| 1- \frac{\bar{\mu}(x-r)} {\bar{\mu}(y-r) }\right| \bar{\mu}(y-r)  k(r) \varphi(x-r) \,dr
\\&\quad+ \int_{y-\varepsilon}^y \left| \bar{\mu}(x-r) -\bar{\mu}(y-r) \right| k(r) \varphi(x-r) \,dr
\\&=\text{V+VI}
\end{align*}
Due to $\varphi\in C[0, T]$ and $\bar{\mu}$ is completely monotone function, using Lebesgue's dominated convergence theorem, we can get that the term V tends to 0, as $x$ goes to $y$. Finally, for the term VI 
\begin{align*}
 &\int_{y-\varepsilon}^y \left| \bar{\mu}(x-r) -\bar{\mu}(y-r) \right| k(r) \varphi(x-r) \,dr
\\&\leq 2 \int_{y-\varepsilon}^y \bar{\mu}(y-r) k(r) \varphi(x-r) \,dr
\\&\leq 2\|\varphi\|_\infty  \int_{y-\varepsilon}^y \bar{\mu}(y-r) k(r) \,dr \xrightarrow[\varepsilon\to 0]{} 0.
\end{align*}
The limit of the end of the calculation holds since $\varphi\in C[0, T]$ and $k$  has no singularity (note that $y\neq 0$) and $\bar{\mu}$ is locally integrable. Thus, the term VI goes to 0. For the case $x<y$, we can use the same strategy. 
\end{proof}
\begin{lemma}\label{ser-cov-14} Let Assumption \ref{ass-a}  and $q=\lim\limits_{x\to 0}\bar{\mu}(x)\int_0^x k(s)\,ds<1$  hold true.  Let $k,\bar\mu$ be a Sonine pair and $\Lscr(\bar\mu, \lambda)=f(\lambda)/\lambda$,\quad $\Lscr(k, \lambda)=1/f(\lambda)$.  Assume that for all  $\varphi \in C[0,T]$  and  for all $x\in [0, T]$,  $ |\varphi(x)|\leq M\int_0^x k(s)\,ds$.  Then for all  $\phi\in C [0,T]$,  $ \sum_{j=1}^\infty \mathcal{K}^j  \varphi \in C[0,T]$ and for all $x\in [0, T]$, $$ |\sum_{j=1}^\infty \mathcal{K}^j\varphi (x) |\leq  \sum_{j=1}^\infty |\mathcal{K}^j \varphi(x)|\leq M \sum_{j=1}^\infty q^j  \int_0^x k(s)\,ds.$$ 
\end{lemma}
\begin{proof} By the proof of Lemma \ref{ope-bou}, $\mathcal{K}\varphi\in C[0,T]$ and $|\mathcal{K}^j\varphi(x)|\leq M q^j \int_0^x k(s) \,ds$. Summing over $j$, we get $$ \left|\sum_{j=1}^\infty \mathcal{K}^j \varphi(x)\right|\leq \sum_{j=1}^\infty\left|\mathcal{K}^j \varphi(x)\right|\leq M \sum_{j=1}^\infty q^j \int_0^x k(s) \,ds .$$  

\noindent The series on the right hand side is uniformly convergent for $x\in [0, T]$. Hence we get $ \sum_{j=1}^\infty \mathcal{K}^j\varphi(x) $ uniformly convergent to a limit in $C[0,T]$.
\end{proof}
 \begin{lemma}\label{es-rei-liu-15} Let Assumption \ref{ass-a}  and $q=\lim\limits_{x\to 0}\bar{\mu}(x)\int_0^x k(s)\,ds<1$  hold true.  Let $k,\bar\mu$ be a Sonine pair and $\Lscr(\bar\mu, \lambda)=f(\lambda)/\lambda,\quad \Lscr(k, \lambda)=1/f(\lambda)$.   Then $$|\rli{0}{x}{f}\varphi |\leq \|\varphi\|_\infty\int_0^x k(s)\,ds \text{\quad for all \quad } \varphi\in C[0, T]. $$
\end{lemma}
\begin{proof}
By the definition of the integral,  we can easily get the following results,
\begin{align*}
|\rli{0}{x}{f}\varphi|&=\int_0^x \varphi(s) k(x-s)\,ds
\\&\leq\int_0^x |\varphi(s)|k(x-s)\,ds
\\&\leq \|\varphi\|_\infty\int_0^x k(s)\,ds \qedhere
\end{align*}
\end{proof}
Next we begin with the censored initial value problem \eqref{ivp-15} with $g\in C[0,T]$ and $\phi(0)\in \real$.
\begin{equation}\label{ivp-15}
\rld{0}{x}{f}\phi=\left\{ \begin{aligned} & g(x)+\phi(x)\bar{\mu}(x), &\quad x>0;\\
&\phi(0), & x=0.
\end{aligned}
\right.
\end{equation}

\begin{theorem}\label{inv-cen-ope-16} Let Assumption \ref{ass-a}  and $q=\lim\limits_{x\to 0}\bar{\mu}(x)\int_0^x k(s)\,ds<1$  hold true.  Let $k,\bar\mu$ be a Sonine pair and $\Lscr(\bar\mu, \lambda)=f(\lambda)/\lambda, \quad \Lscr(k, \lambda)=1/f(\lambda)$. $\phi_0\in \real$,  $g\in C[0,T]$, then there exists a unique function $\phi \in C_{\bar{\mu}}[0,T]$ satisfying \eqref{ivp-15}, and it has the series representation 
\begin{equation}\label{3-19}
\phi(x)-\phi_0=\cei{0}{x}{f}g=\sum_{j=0}^\infty \mathcal{K}^j\rli{0}{x}{f} g.
\end{equation}
\end{theorem}
\begin{remark}\label{re-3-4-18} We see from the proof of Theorem \ref{inv-cen-ope-16} that if 
 $g\in C[0, T]$, then we have $\cei{0}{\cdot}{f}g\in C_{\bar{\mu}}[0, T]$.
\end{remark}
\begin{proof}
Step 1. we show the existence of $\phi$ using Picard iteration.
Set  $\bar{\phi}(x)=\phi(x)-\phi(0)$,  we have the following  equality.
\begin{equation}\label{ivp-17}
\rld{0}{x}{f}\bar{\phi}=\left\{\begin{aligned}&g(x)+\bar{\phi}(x)\bar{\mu}(x), & x>0;\\
&0, & x=0.
\end{aligned}
\right.
\end{equation}
We define the following sequence 
\begin{equation}
\left\{
\begin{aligned}
\bar{\phi}_{n+1}(x)&=\rli{0}{x}{f}\big[g+\bar{\phi}_n\bar{\mu}\big]\\
\bar{\phi}_{0}(x)&=0.
\end{aligned}
\right.
\end{equation}
By Picard iteration we have $$\bar{\phi}_{n+1}(x) =\sum_{j=0}^n\mathcal{K}^j\big[\rli{0}{x}{f}[g]\big].$$
For every $x$ the limit $ \lim_{n\to \infty}\bar{\phi}_{n+1}(x)= \bar{\phi}(x)=\sum_{j=0}^{\infty}\mathcal{K}^j\big[\rli{0}{x}{f}[g]\big] $ exists  due to Lemma \ref{es-rei-liu-15} and Lemma \ref{ser-cov-14}. We get $$\bar{\phi}(x)=\sum_{j=0}^{\infty}\mathcal{K}^j\big[\rli{0}{x}{f}[g]\big] \quad \text {i.e.} \quad \cei{0}{x}{f}g=\sum_{j=0}^\infty \mathcal{K}^j\rli{0}{x}{f} [g].$$

Step  2.  Before we continue our proof, let us show the following auxiliary statement and estimate: 
\begin{lemma}\label{dif-rep-cen}  Under the conditions of Theorem \ref{inv-cen-ope-16},  $\cei{0}{x}{f}g$ can be equivalently represented as
\begin{align}
\cei{0}{x}{f}g &=\rli{0}{x}{f}\left[\bar{\mu} \sum_{j=0}^{\infty}\mathcal{K}^j\left[\bar{\mu}^{-1}g \right]\right]\\&=\sum_{j=0}^{\infty}\mathcal{K}^{j+1}\big[\bar{\mu}(x)^{-1}g(x)\big]
\end{align}
\end{lemma}
\begin{proof}[proof of lemma \ref{dif-rep-cen}] Using the definition of  the above Step 1,  and the definition of the operator $\mathcal{K}$ in Definition \ref{ker-def-1}, we get
\begin{align*}
\cei{0}{x}{f}g &=\sum_{j=0}^{\infty}\mathcal{K}^j\big[\rli{0}{x}{f}[g]\big]=\sum_{j=0}^{\infty}\int_0^xk_j(x,r)\rli{0}{r}{f}[g(r)]\,dr
\\&=\sum_{j=0}^{\infty}\int_0^xk_j(x,r)\mathcal{K}\big[\bar{\mu}(r)^{-1}g(r)\big]\,dr
=\sum_{j=0}^{\infty}\mathcal{K}^{j+1}\big[\bar{\mu}(x)^{-1}g(x)\big]
\\&=\sum_{j=0}^{\infty}\mathcal{K}\left[\bar{\mu}(x)^{-1} \bar{\mu}(x)\mathcal{K}^j\left[\bar{\mu}(x)^{-1}g(x)\right]\right]
=\sum_{j=0}^{\infty}\rli{0}{x}{f}\left[\bar{\mu}\mathcal{K}^j\left[\bar{\mu}^{-1}g \right]\right]
\\&=\rli{0}{x}{f}\left[\sum_{j=0}^{\infty}\bar{\mu} \mathcal{K}^j\left[\bar{\mu}^{-1}g \right]\right].
\qedhere
\end{align*}
\end{proof}
\noindent We will now continue with the proof of Theorem \ref{inv-cen-ope-16}.
Denote $\hat g=\bar{\mu} \sum_{j=0}^{\infty}\mathcal{K}^j\left[\bar{\mu}^{-1}g \right]$, then we have the following estimate:
\begin{equation}\label{3-4-27}
\mathcal{K}^j \left[\bar{\mu}^{-1} g\right](x)\leq q^{j-1} \|g\|_\infty \int_0^x k(r)\,dr.
\end{equation}
We prove the \eqref{3-4-27} by induction.
Indeed, when $j=1$, using the definition of the operator $\mathcal{K}$ we have 
\begin{align*}
\mathcal{K} \bar{\mu}^{-1}(x) g(x)&= \int_0^x \bar{\mu}(r)(\bar{\mu}^{-1} (r)g(r)k(x-r)\,dr
\\&\leq \|g\|_\infty \int_0^x k(r)\,dr.
\end{align*}
We assume that \eqref{3-4-27} holds for some $j\in \mathbb{N}$. For $j \rightsquigarrow j+1$, we obtain
\begin{align*}
\mathcal{K}^{j+1}\bar{\mu}^{-1}(x) g(x)&=\mathcal{K}^j \mathcal{K}\bar{\mu}^{-1}(x) g(x)
\\&\leq \|g\|_\infty \mathcal{K}^j  \int_0^x k(r)\,dr
\\&\leq \|g\|_\infty q^j \int_0^T k(r)\,dr,
\end{align*}
where the second inequality follows by Lemma \ref{ser-cov-14}. This finishes argument induction. 

Step 3. We show $\cei{0}{\cdot}{f}g\in C_{\bar{\mu}}[0,T] $. By Lemma \ref{ser-cov-14}, Lemma \ref{es-rei-liu-15} and \eqref{3-19}, we have $\cei{0}{\cdot}{f}g\in C[0,T]$ since $g$ is continuous. Next we need to show $\cei{0}{\cdot}{f}g\in  C_{\bar{\mu}}(0,T]$. For this, we write  $\cei{0}{\cdot}{f}g=\rli{0}{\cdot}{f}\hat g$,  with $\hat g=\bar{\mu} \sum_{j=0}^{\infty}\mathcal{K}^j\left[\bar{\mu}^{-1}\varphi \right]$ as in Step 2 and show that $\hat g\in C(0, T]\cap L^1(0, T]$ see Theorem \ref{inv-berf-der}. 
In fact, integrability of $ \hat g$ follows from
\begin{align*}
\int_0^T \left|\sum_{j=0}^{\infty}\bar{\mu}(x)\mathcal{K}^j\big[\bar{\mu}(x)^{-1}g(x)\big]\right| \,dx 
&\leq \int_0^T \bar{\mu}(x)\sum_{j=0}^{\infty} \left|\mathcal{K}^j\big[\bar{\mu}(x)^{-1}g(x)\big]\right| \,dx 
\\ &\leq\frac {\|g\|_\infty}{1-q}\int_0^T\int_0^x k(s)\,ds\,\bar{\mu}(x) \,dx
\\ &\leq\frac {\|g\|_\infty}{1-q}\int_0^T\bar{\mu}(x) \,dx\int_0^T k(s)\,ds
\\&<\infty, 
\end{align*}
where the first inequality follows by \eqref{3-4-27}.
To see the continuity of $\hat g$, it suffices to show that $\sum_{j=0}^{\infty}\mathcal{K}^j\left[\bar{\mu}^{-1}g \right]$ is continuous, since $\bar\mu$ is continuous on $C(0, T]$ by Assumption \ref{ass-a}. Indeed, to show that $\sum_{j=0}^{\infty}\mathcal{K}^j\left[\bar{\mu}^{-1}g \right]$ is uniformly convergent, we can use the following estimate:
\begin{align*}
\sum_{j=0}^{\infty}\mathcal{K}^j\left[\bar{\mu}^{-1}g \right]
=\bar{\mu}^{-1}g +\sum_{j=1}^{\infty}\mathcal{K}^j\left[\bar{\mu}^{-1}g \right].
\end{align*}
Clearly $\bar{\mu}^{-1}g\in C(0, T]$. Moreover
\begin{align*}
\sum_{j=1}^{\infty}\mathcal{K}^j\left[\bar{\mu}^{-1}g \right]
=\|g\|_\infty+\|g\|_\infty \sum_{j=1}^{\infty} q^{j-1}\int_0^x k(r)\,dr
\leq \frac{\|g\|_\infty}{1-q} \int_0^T k(r)\,dr.
\end{align*}
Thus, $\sum_{j=1}^{\infty}\mathcal{K}^j\left[\bar{\mu}^{-1}g \right]$  and $\sum_{j=0}^{\infty}\mathcal{K}^j\left[\bar{\mu}^{-1}g \right]$  is continuous and in $C(0, T]$. 
Combining the above results, we can get $\cei{0}{\cdot}{f}g\in C[0,T]$ and $\cei{0}{\cdot}{f}g\in C_{\bar\mu}(0, T]$, then we obtain $\cei{0}{\cdot}{f}g\in C_{\bar\mu}[0, T]$.

Step 4. We show that the solution is unique in $C_{\bar{\mu}}[0,T] $. Let $\phi_1,\phi_2\in C_{\bar{\mu}}[0,T] $ be two solutions to the problem (\ref{ivp-15}). By the linearity of the operator $\ced{0}{\cdot}{f}$, $\psi=\phi_1-\phi_2\in C_{\bar{\mu}}[0,T]$ satisfies the following equation
\begin{equation}\label{ivp-1}
\rld{0}{x}{f}\psi =\left\{\begin{aligned}&\psi(x)\bar{\mu}(x),  \quad x>0;\\
& 0, \quad x=0.
\end{aligned}
\right.
\end{equation}
Using Theorem \ref{inv-berf-der}, Definition \ref{ker-def-1} and Lemma \ref{uni-riem-spac} , apply the inverse operator of $\rld{0}{x}{f}$ on both sides, we get $\psi (x)=\rli{0}{x}{f}\big[\psi \bar{\mu}\big]=\mathcal{K}\psi(x)$. If $\psi\equiv 0$ on $[0,T]$, we get uniqueness.  Assume, to the contrary, that  $\psi\neq 0$.  Because of $\psi (x)=\mathcal{K}\psi(x)$ we have \linebreak $ \int_0^x (\psi(x)-\psi(r))\bar{\mu}(r)k(x-r)\,dr=0$, for all  $x\in [0,T]$. This is impossible: Take $ \xi\in \argmax_{r\in[0,T]}|\psi(r)|$, we have $ \int_0^\xi (\psi(\xi)-\psi(r))\bar{\mu}(r)k(\xi-r)\,dr=0$. This is  impossible, and we have reached  a contradiction.
\end{proof}
\section{Resolvent equation}

To show $\ced{0}{x}{f}$ is a generator of a stochastic process , we will apply the Hille-Yosida theorem, which leads us to solve the resolvent equation \eqref{eigen-equ-1}.
\begin{theorem}\label{lin-cen-ini-val} Let Assumption \ref{ass-a} and $q=\lim\limits_{x\to 0}\bar{\mu}(x)\int_0^x k(s)\,ds<1$ hold true.  Let $k,\bar\mu$ be a Sonine pair and $\Lscr(\bar\mu, \lambda)=f(\lambda)/\lambda, \quad \Lscr(k, \lambda)=1/f(\lambda)$.  Moreover we assume that $f$ is unbounded.  For any $\phi_0\in \real$,  $\lambda\in \real$, the following initial value problem  
\begin{equation}\label{eigen-equ-1}
\left\{
\begin{aligned}
\ced{0}{x}{f}\phi(x)&=\lambda \phi(x), & x\in (0, T]
\\ \phi(x)&=\phi_0, &x=0
\end{aligned}
\right.
\end{equation}
has a unique solution $\phi$ in $C_{\bar{\mu}}[0,T]$ given by $\phi(x)=\phi_0\sum_{j=0}^{\infty}(\lambda \cei{0}{x}{f})^j \mathds{1}$.
\end{theorem}

\begin{proof}
Without loss of generality we assume $\lambda>0$, if $\lambda<0$, set $\Lambda=-\lambda$ repeated the procedure $\lambda>0$. 
The proof relies on an iteration scheme, i.e. we define the following sequence 
\begin{equation}
\left\{
\begin{aligned}
\phi_{n}(x)&=\lambda \cei{0}{x}{f}\phi_{n-1}+\phi_0\\
\phi_{0}(x)&=\phi_0.
\end{aligned}
\right.
\end{equation}
By Picard iteration, we have $$\phi_{n+1}(x) =\sum_{j=0}^{n}(\lambda \cei{0}{x}{f})^j\phi_0.$$

\noindent  We have to show that $ \sum_{j=0}^{\infty}(\lambda \cei{0}{x}{f})^j \mathds{1}$ converges.

\noindent We use induction to show 
\begin{equation}\label{est-gen-ker-3}
(\lambda \cei{0}{x}{f})^j \mathds{1} \leq \left(\frac{\lambda}{1-q}\right)^j \left(\rli{0}{x}{f}\right)^{j-1} \int_0^{\tiny\bullet} k(y)\,dy.
\end{equation}
Recall the definition $\cei{0}{x}{f}\mathds{1}=\sum_{i=0}^\infty \mathcal{K}^i \rli{0}{x}{f} \mathds{1}$ in  Theorem \ref{inv-cen-ope-16} \eqref{3-19}. 
For $j=1$,  we get from  Lemma \ref{ope-bou} and from the the following inequality, 
\begin{align*}
\mathcal{K}^i \rli{0}{x}{f} \mathds{1} \leq q^i \int_0^x k(y)\,dy,
\end{align*}
and from the above  definition of $\cei{0}{x}{f}\mathds{1}$ that $$ \cei{0}{x}{f}\mathds{1}=\sum_{i=0}^\infty \mathcal{K}^i \rli{0}{x}{f} \mathds{1} \leq \sum_{i=0}^\infty q^i \int_0^x k(y)\,dy=\frac{1}{1-q} \int_0^x k(y)\,dy.$$
Using \eqref{est-gen-ker-3} as induction assumption for some $j\in \mathds{N}$, we get  for $j \rightsquigarrow j+1$, that
\begin{align*}
\left(\cei{0}{x}{f}\right)^{j+1}\mathds{1}&\pomu{\eqref{3-19}}{=}{}\cei{0}{x}{f} \left(\cei{0}{x}{f}\right)^{j}\mathds{1}
\\&\pomu{\eqref{3-19}}{\leq}{} \frac{1}{(1-q)^j} \cei{0}{x}{f} \left(\rli{0}{x}{f}\right)^{j-1} \int_0^{\tiny\bullet} k(y)\,dy
\\&\omu{\eqref{3-19}}{=}{}\frac{1}{(1-q)^j} \sum_{i=0}^\infty \mathcal{K}^i \rli{0}{x}{f} \left[\left(\rli{0}{\cdot}{f}\right)^{j-1} \int_0^{\tiny\bullet} k(y)\,dy\right]
\\&\pomu{\eqref{3-19}}{\leq}{}\frac{1}{(1-q)^{j+1}} \left(\rli{0}{x}{f}\right)^j\int_0^{\tiny\bullet} k(y)\,dy. 
\end{align*}
The last inequality follows from the calculation below: 
 \begin{align*}
\mathcal{K}^i \rli{0}{x}{f} \left[\left(\rli{0}{\cdot}{f}\right)^{j-1} \int_0^{\tiny\bullet} k(y)\,dy\right]
&=\mathcal{K}^{i-1} \mathcal{K} \left[\left(\rli{0}{x}{f}\right)^{j}\int_0^{\tiny\bullet}  k(y)\,dy\right]
\\&=\mathcal{K}^{i-1} \left(\rli{0}{x}{f}\right) \bar{\mu}(\cdot) \left[\left(\rli{0}{\cdot}{f}\right)^{j} \int_0^{\tiny\bullet} k(y)\,dy\right]
\\&\stackrel {(*)}{\leq} \mathcal{K}^{i-1} \left(\rli{0}{x}{f}\right)  \left[ \left(\rli{0}{\cdot}{f}\right)^{j} \bar{\mu}(\cdot) \int_0^{\tiny\bullet} k(y)\,dy\right]
\\&\leq q \mathcal{K}^{i-1} \left(\rli{0}{x}{f}\right)^{j+1} \mathds{1}
\\&= q \mathcal{K}^{i-1} \left(\rli{0}{x}{f}\right)^{j} \int_0^{\tiny\bullet} k(y).
\end{align*}
In the step marked by $(*)$ we use the monotonicity of $\bar\mu$.
Repeated  use of  the above calculation, yields
$$\mathcal{K}^i \rli{0}{x}{f} \left[\left(\rli{0}{x}{f}\right)^{j-1} \int_0^{\tiny\bullet} k(y)\,dy\right]\leq q^i \left(\rli{0}{x}{f}\right)^{j} \int_0^x k(y)\,dy.$$
This finishes the proof of \eqref{est-gen-ker-3}. 

Now we show how the assertion of the Theorem follows from \eqref{est-gen-ker-3}.  Taking the  Laplace transform on the right hand side of \eqref{est-gen-ker-3},  we get of the theorem:
\begin{align*}
\Lscr_{x\to s}\left(\left(\rli{0}{x}{f}\right)^{j-1} \int_0^{\tiny\bullet} k(y)\,dy, s\right)
&= \left(\Lscr_{x\to s} \left[k(x), s\right]\right)^{j-1} \Lscr_{x\to s} \left(\int_0^{\tiny\bullet} k(y)\,dy, s\right)
\\&=\frac1{\left(f(s)\right)^{j-1}}\frac1{sf(s)}
\\&=\frac1{s\left(f(s)\right)^j}
\end{align*}
Define $F_n(x)=\sum_{j=0}^n \left(\frac{\lambda}{1-q}\right)^j \left(\rli{0}{x}{f}\right)^{j-1} \int_0^{\tiny\bullet} k(y)\,dy$.
We get 
\begin{align*}
\Lscr_{x\to s}\left[F_n(x), s\right] &\coloneqq \Lscr_{x\to s}\bigg[\sum_{j=0}^n \left(\frac{\lambda}{1-q}\right)^j \left(\rli{0}{x}{f}\right)^{j-1} \int_0^{\tiny\bullet} k(y)\,dy, s\bigg]
 \\&=\sum_{j=0}^n \left(\frac{\lambda}{1-q}\right)^j \Lscr_{x\to s}\left[ \left(\rli{0}{x}{f}\right)^{j-1} \int_0^{\tiny\bullet} k(y)\,dy, s\right]
 \\&=\sum_{j=0}^n \left(\frac{\lambda}{1-q}\right)^j \frac1{s\left(f(s)\right)^j}
 \end{align*}
 If  $s\geq s_0$ is large enough to guarantee that $\lambda/[f(s)(1-q)]<1$, we see that
 $$\lim_{n\to\infty}\Lscr_{x\to s}[F_n(x), s]= \Lscr_{x\to s}[ F(x), s],\quad \forall s\geq s_0.$$
 For all $n, m$ big enough and $s\geq s_0$, we obtain
\begin{align*}
\Lscr_{x\to s_0}\left[ |F_n(x)-F_m(x)|, s_0\right]&=\int_0^\infty e^{-s_0 x}|F_n(x)-F_m(x)|\,dx
\\&=\|F_n-F_m\|_{L^1(e^{-s_0 x}\,dx)}
\end{align*}
Using the diagonal method, we get a subsequence $F_{n_l}$ which converges pointwise almost everywhere to a function $F$. Without loss of generality we can assume that $F_{n_l}(T)$ converges to $F(T)$, otherwise we take another value for $T$.  Assuming $m>n$ and the positive perversity of $\rli{0}{x}{f}$ , we have 
 \begin{align*}
 \sup_{x\in [0, T]}\left|F_n(x)-F_m(x)\right|&=\sup_{x\in [0, T]}\left|\sum_{n+1}^m \left(\frac{\lambda}{1-q}\right)^j  \left(\rli{0}{x}{f}\right)^{j-1} \int_0^{\tiny\bullet} k(s)\,ds\right|
 \\&\leq \left|\sum_{n+1}^m\left(\frac{\lambda}{1-q}\right)^j  \left(\rli{0}{T}{f}\right)^{j-1} \int_0^{\tiny\bullet} k(s)\,ds\right|
 \\&= \left|F_n(T)-F_m(T)\right|
 \\&\leq\left|F_{n_l}(T)-F_{m_l}(T)\right|\xrightarrow[n, m\to \infty]{}0.
 \end{align*}
 where $n_l\leq n+1<m\leq m_l$ are such that $n_l, m_l$ goes to $\infty$. Now we define $$G_n(x)=\sum_{j=0}^n (\lambda \cei{0}{x}{f})^j \mathds{1}.$$
Using \eqref{est-gen-ker-3}, we get $G_n(x)\leq F_n(x)$ and furthermore, assuming $m>n$ we obtain  uniformly for all $x\in [0, T]$,
\begin{align*}
\left|G_n(x)-G_m(x)\right|&=\left| \sum_{n+1}^m (\lambda \cei{0}{x}{f})^j \mathds{1}\right|
\\&\leq \left|\sum_{n+1}^m \left(\frac{\lambda}{1-q}\right)^j  \left(\rli{0}{T}{f}\right)^{j-1} \int_0^{\tiny\bullet} k(s)\,ds\right|
\\&=\left|F_n(T)-F_m(T)\right|\xrightarrow[n, m\to \infty]{}0.
\end{align*}
From this we get that  $G_n$ converges to $G$ (locally) uniformly.  Thus,  $$G(x)=\phi_0\sum_{j=0}^{\infty}(\lambda \cei{0}{x}{f})^j \mathds{1} \quad \text{and} \quad G\in C[0, T].$$ Note that from Remark \ref{re-3-4-18} $$G(x)=\phi_0\sum_{j=1}^{\infty}(\lambda \cei{0}{x}{f})^j \mathds{1}+\phi_0=\phi_0\lambda \cei{0}{x}{f} \big[\sum_{j=0}^{\infty}(\lambda \cei{0}{x}{f})^j \mathds{1}\big]+\phi_0,$$ 

\noindent by Theorem \ref{inv-cen-ope-16}, $G\in C_{\bar\mu}[0, T]$ and it solves \eqref{eigen-equ-1}.
\bigskip

Next we show the uniqueness of the solution to \eqref{eigen-equ-1}.  Assume $\phi_1, \phi_2$ are two solutions of \eqref{eigen-equ-1}. By the linearity of the fractional derivative we have that $\phi=\phi_1-\phi_2$ is the solution of the following equation 
 \begin{equation*}
 \left\{
\begin{aligned}
\ced{0}{x}{f}\phi&=0, &x\in (0, T]
\\ \phi(x)&=0, &x=0
\end{aligned}
\right.
\end{equation*}
Using the strategy of Theorem \ref{inv-cen-ope-16} Step 4, we get $\phi\equiv 0$, i.e. uniqueness.
\end{proof}
For inhomogeneous resolvent equation \eqref{in-hom-ivp}, we have the following result. 
\begin{theorem}\label{inh-ivp} Let  Assumption \ref{ass-a} and $q=\lim_{x\to 0}\bar{\mu}(x)\int_0^x k(s)\,ds<1$ hold true.  Let $k,\bar\mu$ be a Sonine pair and $\Lscr(\bar\mu, \lambda)=f(\lambda)/\lambda, \Lscr(k, \lambda)=1/f(\lambda)$. For any $\phi_0\in \real$,  $\lambda\in \real$, $g\in C[0,T]$, the linear initial value problem 

\begin{equation}\label{in-hom-ivp}
\left\{
\begin{aligned}
\ced{0}{x}{f}\phi(x)&=\lambda \phi(x)+g(x), x\in (0, T]
\\ ~~~~~~~~~~\phi(x)&=\phi_0, ~~~~~~~~~~~~~~x=0,
\end{aligned}
\right.
\end{equation}
has a unique solution in $C_{\bar{\mu}}[0,T]$ given by $$\phi(x)=\phi_0 \sum_{j=0}^{\infty}(\lambda \cei{0}{x}{f})^j \mathds{1}+\sum_{j=0}^{\infty}\lambda^j (\cei{0}{x}{f})^{j+1}g.$$
\end{theorem}
\begin{proof} Using the proof of Theorem \ref{lin-cen-ini-val}, we can show  $$\sum_{j=0}^{\infty}\lambda^j (\cei{0}{x}{f})^{j+1} g\quad \text {and} \quad \sum_{j=0}^{\infty}\lambda^j (\cei{0}{x}{f})^{j+1} \mathds{1}$$ are (locally) uniformly convergence in $C[0,T]$.  Applying  $\cei{0}{x}{f}$  on both side of \eqref{in-hom-ivp} gives after inserting the concrete form of $\phi$, 
\begin{align*}
\lambda \cei{0}{x}{f}\phi+\cei{0}{x}{f}g&\pomu{(*)}{=}{}\lambda \phi_0 \cei{0}{x}{f}\sum_{j=0}^{\infty}(\lambda \cei{0}{x}{f})^j 1+\lambda \cei{0}{x}{f} \sum_{j=0}^{\infty}\lambda^j (\cei{0}{x}{f})^{j+1} g+\cei{0}{x}{f}g
\\&\omu{(*)}{=}{}  \phi_0 \sum_{j=0}^{\infty}(\lambda \cei{0}{x}{f})^{j+1} 1+ \sum_{j=0}^{\infty}\lambda^{j+1} (\cei{0}{x}{f})^{j+2} g+\cei{0}{x}{f}g
\\&\pomu{(*)}{=}{}  \phi_0+ \sum_{j=1}^{\infty}(\lambda \cei{0}{x}{f})^{j+1} 1+ \sum_{j=1}^{\infty}\lambda^{j+1} (\cei{0}{x}{f})^{j+2} g
\\&\pomu{(*)}{=}{}\phi(x)-\phi_0
\end{align*}
 In the step marked by (*) we use the continuity of $\cei{0}{x}{f}$ which is clear from its construction using the Banach fixed--point theorem.
Using Theorem \ref{inv-cen-ope-16}, $\phi$  solves \eqref{in-hom-ivp} and this solution is unique.
\end{proof}
\section{Stochastic process}\label{con-cen-pro}
The starting point of the censored process is always assumed to be some fixed $x>0$. We will define the censored decreasing subordinator $S^c$ by piecing together. This construction guarantees that we will get a right continuous strong Markov process by using Theorem 1.1 in \cite{1966_Nobuyuki}. The construction is : run $x-S^1_t$ until $\tau_1$, the time when $S^1$ first exits 0, where $S^1$ is an increasing subordinator  given by $f$ and  starting at 0. If $x-S^1_{\tau_1}$ is less or equal than 0, then kill the process, 
then construct an i.i.d copy $-S^2$ of $-S^1$ and start this at $x-S_{\tau_1-}$, i.e. we piece together i.i.d copies at those points in time where the process would leave $(0,\infty)$. Now repeat the same procedure for at most countably many times.  
Thus, we prove the above construction $S_t^c$ is a Feller process and which has  generator given by \eqref{cen-gen-8}, i.e $\ced{0}{\cdot}{f}$.

More formally, set $\tau_j=\sum_{i=0}^{j}\sigma_i$, we define the censored decreasing process $S^c$ as 
\[
S_t^c =
\left\{ 
\begin{array}{ll}
    x-S_t^1, & 0\leq t<\sigma_1,  \quad j=1, \\
    S^c_{\tau_{j-1} -}-S^j_{t-\tau_{j-1}}, & \tau_{j-1} \leq t<\tau_j , \quad j\geq 2, \\
    \partial, & t\geq \tau_j,
\end{array}
\right.
\]
where $\partial$ is the grave yard (cemetery point) and 
\begin{alignat*}{2}
 \sigma_j = \left\{
\begin{aligned} & 0   \quad &\mbox{if }  j=0~, \\ & \inf\{t>0: S_{t+\sigma_{j-1}}^c \leq 0\} \quad &\mbox{ if } j \in\mathbb{N}, 
 \end{aligned}
\right.
\end{alignat*}

\noindent  where $\{S^j\}_{j\in \mathbb{N}}$ is an i.i.d. collection of an increasing subordinators given by $f$.  Write $E_j$ for that the generalized inverse of  $y-S_{\tiny\bullet}^j$  i.e.  $E_j(y)=\inf\{s:y-S^j_s<0\}$, for all $j\in \mathbb{N}$, $y>0$. Further, $U(dx)=k(x)\,dx$ is the potential measure of $S^1$.
\begin{theorem}\label{str-mar-pro} The above construction gives a  strong Markov process $S_{\tiny\bullet}^c$.
\end{theorem}

\begin{proof}
Set $A_j=\{\tau_{j-1}\leq \eta <\tau_j\}$ and $B_i=\{\tau_{i-1}\leq\eta+s<\tau_i\}$. Clearly, $ \linebreak \Omega=\cup_{j=1}^\infty\cup_{i=1}^\infty A_j\cap B_i$ and the union is a union of mutually disjoint sets. Therefore it is enough to consider \eqref{4-1-2} on $A_j\cap B_i$.
\begin{equation}\label{4-1-2}
\Ee^x[\phi(S^c_{\eta+s})|\mathcal{F}_\eta]=\Ee^{S^c_\eta}[\phi(S^c_{s})].
\end{equation}
Case 1 If $i=j$ we have 
\begin{align*}
&\Ee^x\left[\phi(S^c_{\eta+s}) \mathds{1}_{A_j}\mathds{1}_{B_i} |\mathcal{F}_\eta\right]
\\&\pomu{(*)}{=}{}\mathds{1}_{A_j} \Ee^x\left[\phi\left(S^c_{\tau_{j-1}-}-S^j_{\eta-\tau_{j-1}}+(S^j_{\eta-\tau_{j-1}}-S^j_{\eta+s-\tau_{j-1}})\right) \mathds{1}_{B_i} |\mathcal{F}_\eta\right]
\\&\omu{(*)}{=}{} \Ee^x\bigg[\phi\big(S^c_{\eta}+\underbrace{(S^j_{\eta-\tau_{j-1}}-S^j_{\eta+s-\tau_{j-1}})}_{\sim S^c_s=-S^1_s}\big) \mathds{1}_{B_i}\mathds{1}_{A_j} |\mathcal{F}_\eta\bigg]
\\&\omu{(*)}{=}{} \Ee^{S^c_\eta}\left[ S_s^c \mathds{1}_{B_i}\mathds{1}_{A_j}\right].
\end{align*}
In the steps  marked with (*) we use the following facts
\begin{enumerate}[(a)]
\item $\Ee^x\left[g(X, Y)|\mathcal{F}\right]=\Ee^x\left[g(z, Y)\right]|_{z=X}$ if $g$ is bounded and measurable and $X$ is $\mathcal{F}$--measurable and $Y$ is independent of $\mathcal{F}$.
\item $S^c_{\tau_{j-1}-}-S^j_{\eta-\tau_j}=S^c_\eta$ on $A_j$.
\item $S^j_{\eta-\tau_{j-1}}-S^j_{\eta+s-\tau_{j-1}}\sim -S^j_s\sim -S^1_s\sim S^c_s$ on $A_j\cap B_i$.
\end{enumerate}

Case 2 If $j\leq i$ we have 
\begin{align*}
&\Ee^x[\phi(S^c_{\eta+s}) \mathds{1}_{A_j}\mathds{1}_{B_i} |\mathcal{F}_\eta]
\\&\pomu{tower}{=}{}\mathds{1}_{A_j}\Ee^x[\phi(S^c_{\eta+s}) \mathds{1}_{B_i} |\mathcal{F}_\eta]
\\&\omu{tower}{=}{} \mathds{1}_{A_j}\Ee^x\left\{\Ee^x\left[\phi(S^c_{\eta+s}) \mathds{1}_{B_i}|\mathcal{F}_{\tau_{i-1}} \right] |\mathcal{F}_\eta\right\}
\\&\pomu{tower}{=}{} \mathds{1}_{A_j}\Ee^x\left\{\Ee^{S^c_{\tau_{i-1}}}\left[\phi(-S^c_{\eta+s-\tau_{i-1}}) \mathds{1}_{B_i} \right] |\mathcal{F}_\eta\right\},
\end{align*}

\noindent where in the last equality, we use similar steps to Case 1.

\noindent We may iterate this argument, conditioning (tower property) on $\mathcal{F}_{\tau_{i-2}}$. Since $\sigma_l=\tau_l-\tau_{l-1}$, we have 
\begin{align*}
&\Ee^x\left[\phi(S^c_{\eta+s}) \mathds{1}_{A_j}\mathds{1}_{B_i} |\mathcal{F}_\eta\right] \\
&\pomu{as in Case 1}{=}{}\mathds{1}_{A_j} \Ee^x\left[\Ee^{S^c_{\tau_{k-2}}}\left[\mathds{1}_{B_i} \Ee^{S^{i-1}_{\sigma_{k-1}-}}\left[\phi(-S^k_\rho)\right]_{\rho=\eta+s-\tau_{i-1}}\right]|\mathcal{F}_\eta\right]
\\&\pomu{as in Case 1}{=}{}\mathds{1}_{A_j} \Ee^x\left[\mathds{1}_{B_i} \Ee^{S^c_{\tau_j}}\left[\phi(S^{j+1}_{\sigma_{j+1}-}+...+ S^{i-1}_{\sigma_{i-1}-} -S^k_\rho)\right]_{\rho=\eta+s-\tau_{i-1}}|\mathcal{F}_\eta\right]
\\&\omu{as in Case 1}{=}{}\Ee^{S_\eta^c}\left[\mathds{1}_{A_j}\mathds{1}_{B_i} \phi(S_s^c)\right].
\end{align*}
Since on $A_j\cap B_i$, $S^c_s$ and $S^j_{\tau_j-\eta}+S^{j+1}_{\sigma_{j+1}-}+...+S^i_{\eta+s-\tau_{i-1}}$ have the same law.
\end{proof}
\begin{remark} Using \cite{1988_Sharpe} Theorem 14.8, we can show the censored process is a strong Markov process. Take transfer kernel $K(S_{\tau_j}^c, dy)=\delta_{S_{\tau_j-}^c}(dy)$, where $\delta$ is  the $\delta$--distribution.
\end{remark}
Denote by $U(dy)=k(y)\, dy$ the potential measure, $1/f(\lambda)=\Lscr(k, \lambda)=\Lscr(U(dx), \lambda)$, see Remark 1.5.11 in \cite{2023_Li}
 and $k_1(x, r)=\bar\mu(r)k(x-r)$ see Definition \ref{ker-def-1}.   
\begin{lemma}\label{pro-con-pro} For any $x>0$ and $j\in \mathbb{N}$,  assuming $S_0^c=x$, we have 
\begin{enumerate}[(i)]
\item $\Ee^x[\tau_j]<\infty$, $\mathbb{P}[S_{\tau_j}^c\in (0,x)]=1$ and $S_{\tau_j}^c$ has the density $k_j(x,\cdot)$, as defined in \eqref{kernel-def};
\item  $\Ee^x [\sigma_{j+1}]=\Ee[E_{j+1}(S_{\tau_j}^c)]=\int_0^x U(y)k_j(x,y)\,dy$;
\item $\mathbb{P}^x[\lim_{j\to\infty}\tau_j<\infty]=1$ and $\mathbb{P}^x[S^c_{\lim_{j\to\infty}\tau_j-}=0]=1$.
\end{enumerate}
\end{lemma}
\begin{proof}
 (i)  We use induction. If $j=1$ we have $$\Ee[\tau_1(y)]=\Ee[E_1(y)]=U(y)<\infty,$$ where we use $E_1(y)=\inf\{s:y-S^1_s<0\}$ and $\Ee[E_1(y)]=U(y)$. 
 From  \cite{1996_Bertoin}  Chapter 3, Proposition 2 we know that  $$\mathbb{P}(x-S_{\sigma_1-}\in dy, x-S_{\tau_1}\in dz)=U(dy)\mu(dz-y),$$ where $\tau(x)=\inf\{t:S_t\geq x\}.$ As $\tau_1=\tau(x)$ under $\mathbb{P}^x$, we get for $0\leq a\leq x$,
\begin{align*}
\Pp^x(S^c_{\tau_1-}\in (a, x])
&=\Pp(x-S^c_{\tau(x)-}\in (a, x])
\\&=\Pp(x-S^c_{\tau(x)-}\in (a, x], S_{\tau(x)}\geq x)
\\&=\Pp(S^c_{\tau(x)-}\in ([0,, x-a), S_{\tau(x)}\geq x)
\\&=\int_0^{x-a}\bar\mu(x-y)k(y)\,dy.
\end{align*}
This shows that $k_1(x,r)=\bar\mu(x-y)k(y)$ is the density of $S^c_{\tau_1-}$ under $\Pp^x$.

Assume $j\geq 1$, $\tau_j<\infty, S_{\tau_j}^c>0$ and $S_{\tau_j}^c$ independent of $S^{j+1}$. According to the strong Markov property,  we have 
\begin{align*}
\sigma_{j+1}&=\inf \{r>0: S^c_{\tau_j}<S^{j+1}_{r}\}
\\&=E_{j+1}(S^c_{\tau_j})
\end{align*}
Using $S^c_{\tau_j}<x$, we can get the following estimate
$$\Ee[\tau_{j+1}]=\Ee[E_{j+1}(S^c_{\tau_j})]+\Ee[\tau_j]\leq \Ee[E_{j+1}(x)]+\Ee[\tau_j]<\infty.$$  

Next we show $S_{\tau_{j+1} }^c\in (0,x)$. Using the expression of $\sigma_{j+1}$, we have $$S_{\tau_{j+1}  }^c=S_{\tau_j}^c - S^{j+1}_{\sigma_{j+1}-}=S_{\tau_j}^c -S^{j+1}_{E_{j+1}(S^c_{\tau_j})-}\in (0, S^c_{\tau_j})\subset (0,x).$$
Finally, we show $S_{\tau_{j+1} }^c$ has density $k_{j+1}(x,\cdot)$. For any bounded measurable $\phi$, we have 
\begin{align*}
\Ee^x \left[\phi(S_{\tau_{j+1} }^c)\right]&\pomu{Fubini}{=}{}\Ee^x \left[\phi(S_{\tau_j }^c - S^{j+1}_{E_{j+1}(S^c_{\tau_j})-})\right]
\\&\pomu{Fubini}{=}{}\Ee\left[\int_0^x\phi(y-S^{j+1}_{E_{j+1}(S^c_{\tau_j })-})\right]\mathbb{P}(S_{\tau_j }^c\in dy)
\\&\pomu{Fubini}{=}{}\Ee\left[\int_0^x\phi(y-S^{j+1}_{E_{j+1}(S^c_{\tau_j })-})\right]k_j(x,y)\,dy
\\&\pomu{Fubini}{=}{}\int_0^x\Ee^y\left[\phi(-S^{j+1}_{E_{j+1}(S^c_{\tau_j})-})\right]k_j(x,y)\,dy
\\&\pomu{Fubini}{=}{}\int_0^x\int_0^y \phi(z)k_1(y,z)\,dz \,k_j(x,y)\,dy
\\&\pomu{Fubini}{=}{}\int_0^x \phi(z)\int_z^x k_j(x,y)k_1(y,z)\,dy\,dz
\end{align*}
The fourth equality holds because $\{-S^j\}_{j\in \mathbb{N}}$ is an i.i.d. collection of subordinators and  $-S^{j+1}_{E_{j+1}(S^c_{\tau_j})}\stackrel{\Pp^y}{\sim} -S^1_{E_1}(y)$  has density $k_1(y,\cdot)$ and $S^c_{\sum_{i=1}^{j} \sigma_i }$ is independent of $S^{j+1}$. Using Definition \ref{ker-def-1}, we know that $S_{\sum_{i=1}^{j+1} \sigma_i }^c$ has density $k_{j+1}(x,\cdot)$.

(ii) Using the results of (i),  we have $\Ee[\sigma_{j+1}]=\Ee[E_{j+1}(S^c_{\tau_j })]$ and  $S^c_{\tau_j}$ is independent of $S^{j+1}$. By the definition of $E_j$ we have,
\begin{align*}
\Ee^x\left[\sigma_{j+1}\right]&=\Ee^x\left[E_{j+1}(S^c_{\tau_j })\right]
\\&=\Ee\left[\int_0^xE_{j+1}(y)k_j(x,y)\,dy\right]
\\&=\int_0^x\Ee\left[E_{j+1}(y)\right]k_j(x,y)\,dy
\\&=\int_0^xU(y)k_j(x,y)\,dy.
\end{align*}

(iii)  From the Markov inequality $$n\mathbb{P}(\tau_\infty>n)\leq \Ee^x[\tau_\infty], \quad n \to \infty.$$ We see that it is enough to show that $\Ee^x(\tau_\infty)<\infty$ to get $\Pp(\tau_\infty<\infty)$.
Now by Beppo Levi 
\begin{align*}
\Ee^x (\tau_\infty)&=\lim_{j\to\infty}\Ee\left[\sum_{i=1}^{j} \sigma_i \right]=\lim_{j\to\infty}\sum_{i=1}^j\Ee[\sigma_i]
\\&=\lim_{j\to\infty}\sum_{i=1}^j\int_0^xU(y)k_{i-1}(x,y)\,dy
\\&=\sum_{i=1}^\infty \int_0^xU(y)k_{i-1}(x,y)\,dy
\\&=\sum_{i=1}^\infty \mathcal{K}^{i-1}U(x)<\infty,
\end{align*}
For the last inequality, we use the definition of $U(y)=\int_0^yk(x)\,dx$ and that fact that $U$ satisfies the condition of Lemma \ref{ser-cov-14}.

Next we show $\mathbb{P}^x[S^c_{\lim_{j\to\infty}\tau_j-}>0]=0$.  $$\mathbb{P}^x\left[S^c_{\lim_{j\to\infty}\tau_j-}>0\right]\leq \sum_{n=1}^\infty \mathbb{P}\left[S^c_{\lim_{j\to\infty}\tau_j-}>\frac1{n}\right], $$
and for each $n\in \mathbb{N}, $, using the following results $\left\{S^c_{\lim_{j\to\infty}\tau_{j+1}}>\frac1{n}\right\}\subset \left\{S^c_{\tau_j}>\frac1{n}\right\}$, $$\mathbb{P}^x\left[S^c_{\lim_{j\to\infty}\tau_j-}>\frac1{n}\right]\leq\mathbb{P}^x\left[\cap_{j=1}^\infty\{S^c_{\tau_j}>\frac1{n}\}\right]=\lim_{j\to\infty}\mathbb{P}^x\left[S^c_{\tau_j}>\frac1{n}\right].$$
  
 Let $F$ be a strictly increasing function, e.g. $F(y)=\int_0^y k(s)\,ds$, $F(0)=0$,  then  by Chebyshev's inequality, 
\begin{align*}
\mathbb{P}^x\left[S^c_{\tau_j}>\frac1{n}\right]&=\mathbb{P}^x\left[F(S^c_{\tau_j})>F\left(\frac1{n}\right)\right]
\\&\leq \frac1{F(\frac1{n})}\Ee\left[F(S^c_{\tau_j})\right]
\\&=\frac1{F\left(\frac1{n}\right)}\int_0^x k_j(x,y)\int_0^y k(s)\,ds\, dy \xrightarrow[j\to \infty]{}0,
\end{align*} 
Indeed, since $\sum_{j=0}^\infty\int_0^x k_j(x,y)\int_0^y k(s)\,ds\, dy$ is finite,  by the results of Lemma \ref{ser-cov-14}, we obtain $\int_0^x k_j(x,y)\int_0^y k(s)\,ds\, dy$ goes to 0, as $j$ goes to $\infty$.
\end{proof}

\section{Probabilistic representation}
Let $(S_t)_{t\geq 0}$ be the subordinator corresponding to the Bernstein function $f$.  As before, $U(dx)$ is the potential measure i.e. $1/f(\lambda)=\Lscr(k, \lambda)=\Lscr(U(dx), \lambda)$ and $\tau_1=\tau(x)$ is the first time $S_t$ cross the level $x$. Again we assume that  Assumption \ref{ass-a} from Section 2 holds.
\begin{lemma}\label{pro-rep-pot}  Assume $S_t^1$ has  transition probability $\mathbb{P}^x(-S_t^1 \in dy)=p_t(x-dy)$, we have the following equality:
\begin{equation}
\rli{0}{x}{f} g(x)=\Ee^x\left[\int_0^{\tau_1} g(-S_s^1)\,ds\right]
\end{equation}
\begin{proof} The results follow from a direct calculation, we have 
\begin{align*}
\Ee^x\left[\int_0^{\tau_1} g(-S_s^1)ds\right]&\pomu{\eqref{pot-rep}}{=}{}\Ee^x\left[\int_0^\infty g(-S_s^1) 1_{\{s\leq \tau_1\}}\,ds\right]
\\&\pomu{\eqref{pot-rep}}{=}{}\Ee^x\left[\int_0^\infty g(-S_s^1) 1_{\{x-S_s^1\leq 0\}}\,ds\right]
\\&\pomu{\eqref{pot-rep}}{=}{}\int_0^\infty\int_0^x  g(y) \mathbb{P}^x(-S_s^1 \in \,dy) \,ds
\\&\pomu{\eqref{pot-rep}}{=}{}\int_0^x \int_0^\infty g(y)\, p_s(x-dy) \,ds
\\&\pomu{\eqref{pot-rep}}{=}{}\int_0^x g(y)\, U(x-dy). 
\end{align*}

Since the Sonine pair $(\bar\mu, k)$ consists of functions because of Assumption \ref{ass-a}, we have $U(dx)=k(x)\,dx$, so 
$$\Ee^x\left[\int_0^{\tau_1} g(-S_s^1)ds\right]=\int_0^x g(y)k(x-y) \,dy.$$ Using the definition of the Riemann--Liouville integral  we get 
\begin{gather*}
\rli{0}{x}{f} g(x) = \Ee^x \left[ \int_0^{\tau_1} g(-S_s^1) \, ds \right]. \qedhere
\end{gather*}
\end{proof}
\end{lemma}
\begin{theorem}\label{pot-rep-cen-int} If $g$ satisfies the conditions of Theorem \ref{inv-cen-ope-16},  we have the following  representation
\begin{equation}
\cei{0}{x}{f}g(x)=\sum_{j=0}^\infty \Ee^x \left[\rli{0}{S^c_{\tau_j}}{f}g(S^c_{\tau_j})\right]= \Ee^x \left[\int_0^{\tau_\infty}g(S^c_s)\,ds\right],
\end{equation}
where $\tau_1, \tau_2....$ are the stopping times from the construction of Section 6.
\end{theorem}
\begin{proof} First, we assume $g\geq0$. Making the change of variable $s-\tau_j=u$, we have 
\begin{align*}
\Ee^x\left[\int_0^{\tau_\infty}g(S^c_s)\,ds\right]&\pomu{(*)}{=}{}\Ee^x\sum_{j=0}^\infty \left[\int_{\tau_j}^{\tau_{j+1}}g(S^c_s)\,ds\right]
\\&\pomu{(*)}{=}{}\sum_{j=0}^\infty \Ee^x \left[\int_0^{\tau_{j+1}-\tau_j}g(S^c_{u+\tau_j})\,du\right]
\\&\pomu{(*)}{=}{}\sum_{j=0}^\infty \Ee^x \left[\int_0^{E_{j+1}(S^c_{\tau_j})}g(S^c_{u+\tau_j})\,du\right]
\\&\pomu{(*)}{=}{}\sum_{j=0}^\infty \Ee^x \left[\int_0^{E_{j+1}(S^c_{\tau_j})}g(S^{j+1}_u+S_{\tau_j}^c)\,du\right]
\\&\pomu{(*)}{=}{}\sum_{j=0}^\infty \Ee^x \left[\Ee^ {S_{\tau_j}^c} \left[\int_0^{E_{j+1}(S^c_{\tau_j})}g(S^{j+1}_u)\,du \right]\right]
\\&\omu{(*)}{=}{}\sum_{j=0}^\infty \Ee^x\left[\rli{0}{S^c_{\tau_j}}{f} g(S^c_{\tau_j})\right]
\\&\pomu{(*)}{=}{}\sum_{j=0}^\infty \mathcal{K}^j \rli{0}{x}{f} g(x) 
\\&\pomu{(*)}{=}{} \cei{0}{x}{f}g(x). 
\end{align*}
In the equality marked by (*) we use Lemma \ref{pro-rep-pot}.
\end{proof}
Recall that $C_\infty(0, T]=\overline{C_c(0, T]}^{\|\cdot\|_\infty}=\{u\in C(0, T]: u(0+)=0\}$. In this context,  the operator $-\ced{0}{\tiny\bullet}{f}$ gives a stochastic process by Hille--Yosida--Ray theorem. 
\begin{theorem}\label{fell-pro} For any $T>0$, the process $S_t^c $ gives a Feller semigroup on $C_\infty(0,T]$, whose generator is $(-\ced{0}{\tiny\bullet}{f}, \cei{0}{\tiny\bullet}{f}C_\infty(0,T])$.
\end{theorem}
\begin{proof} 
Recall the definition of the Hille--Yosida--Ray theorem  
and Theorem \ref{str-mar-pro}. 
We know that $S^c_\cdot$ is a Markov process, so  $P_t^c\phi(x)=\Ee^x[\phi(S_t^c)]$ is a positivity preserving contraction semigroup on the Borel--measurable functions $\mathcal{B}[0,T]$. If we can show  that $P_t^c$ is a Feller operator i.e. $ P_t^c: C_\infty(0,T]\to C_\infty(0,T]$ and  $t\mapsto P_t^c \phi $ is strongly continuous, then we can show $P_t^c$ is a Feller semigroup.

First we show  $P_t^c $ is strongly continuous  on $C_\infty(0,T]$. Assume $\phi\in C_\infty(0,T]$, and define $\phi(x)=\phi(0+), x\in (-\infty,0]$, 
\begin{align*}
|P_t^c\phi(x)-\phi(x)|&=\left|\Ee^x[\phi(S^c_t)]-\phi(x)\right|
\\&=\left|\Ee^x[\phi(S^c_t)]-\Ee^x [\phi(x)]\right|
\\&=\left|\Ee^x\left[\phi(S^c_t)(1_{\{t<\tau_1\}}+1_{\{t\geq \tau_1\}})\right]-\Ee^x \left[\phi(x)(1_{\{t<\tau_1 \}}+1_{\{t\geq \tau_1\}})\right]\right|
\\&\leq \left|\Ee^x\left[(\phi(S^c_t)-\phi(x))1_{\{t<\tau_1\}}\right]\right|+\left|\Ee^x \left[(\phi(S^c_t)-\phi(x))1_{\{t\geq \tau_1\}}\right]\right|
\\&=\left|\Ee^ 0 \left[(\phi(S^1_t+x)-\phi(x))1_{\{t<\tau_1\}}\right]\right|+\left|\Ee^x \left[(\phi(S^c_t)-\phi(x))1_{\{t\geq \tau_1\}}\right]\right|.
\end{align*}
According to the  construction of  Section \ref{con-cen-pro}, we can get that the first summand vanishes uniformly in $x$ as $t\to 0$ because $S^1_t$ is a subordinator.  For the second one we have 
\begin{align*}
\left|\Ee^x \left[(\phi(S^c_t)-\phi(x))1_{\{t\geq \tau_1\}}\right]\right|&\leq 2\|\phi\|_{C[0,x]}\mathbb{P}^0[t\geq \tau_1]
\\&\leq 2\|\phi\|_{C[0,x]}\mathbb{P}^0[t\geq E_1(x)]
\\&\leq 2\|\phi\|_{C[0,x]}\mathbb{P}^0[x-S^1_t\leq 0]
\\&\leq 2\|\phi\|_{C[0,x]}\mathbb{P}^0[x\leq S^1_t].
\end{align*}
 Because of $\phi(0+)=0$, for any $\varepsilon>0$ there exist some   $ \delta$ such that for $0\leq x\leq \delta$, \linebreak $\|\phi\|_{C[0,x]}\leq \varepsilon$, and we have $ \mathbb{P}[x\leq S^1_t]\leq 1$.  If $t\to 0$ and $x>\delta$ we have $\mathbb{P}[x\leq S^1_t]\to 0$. Thus we have the following estimate as $ t\to 0$,
\begin{alignat*}{2}
\|\phi\|_{C[0,x]}\mathbb{P}[x\leq S^1_t] \leq \left\{
\begin{aligned} &\varepsilon, ~~~~~~~~ 0\leq x\leq \delta, \\&\varepsilon \|\phi\|_{C[0,T]} , ~~\delta<x\leq T. 
 \end{aligned}
\right.
\end{alignat*}
This proves that  $P_t^c$ is strongly continuous on $C_\infty(0, T]$. 

Next, we show $P_t^c : C_\infty(0,T]\to C_\infty(0,T]$. We begin by showing that $(P_t^c \phi)(0+)=0$. Clearly $\phi\in C_\infty(0, T]$ satisfies $\phi(0+)=0$, thus our previous calculations show 
\begin{align*}
\left| P_t^c \phi(x)-\phi(x)\right|\leq \left|\Ee^0\left[(\phi(S^1_{t}+x)-\phi(x)) \mathds{1}_{\{t\leq \tau_1\}}\right]\right| +2\|\phi\|_{C[0, x]}\Pp^0(x\leq S_t^1).
\end{align*}
Since $S_t^1\leq x$ if $t\leq \tau_1=\tau(x)$ we get $$\left| P_t^c \phi(x)-\phi(x)\right|\leq 4\|\phi\|_{C[0, x]}\to 0, \quad \text{as} \quad x\to 0. $$
Now we show that $x\mapsto P_t^c\phi(x)$ is continuous. Pick $\phi\in C_\infty(0, T]$ and assume, without loss of generality, that $0\leq x\leq y$. Let $h>0$ be small and write ${}^{z}S^c_t$ for the process $S^c_t$ with $S^c_0=z$. We have
\begin{align*}
&\left|P_t^c\phi(x)-P_t^c\phi(y)\right|\\&=\left|\Ee^x\phi(S^c_t)-\Ee^y\phi(S_t^c)\right|
\\&=\left|\Ee\left[\mathds{1}_{\{{}^{x}S^c_t-{}^{y}S^c_t\leq h\}}\left[\phi({}^{x}S^c_t)-\phi({}^{y}S^c_t)\right]\right]+\Ee\left[1_{\{{}^{x}S^c_t-{}^{y}S^c_t>h\}}\left[\phi({}^{x}S^c_t)-\phi({}^{y}S^c_t)\right]\right]\right|
\\&\leq \left|\Ee\left[\mathds{1}_{\{{}^{x}S^c_t-{}^{y}S^c_t\leq h\}}\left[\phi({}^{x}S^c_t)-\phi({}^{y}S^c_t)\right]\right]\right|+\left|\Ee\left[1_{\{{}^{x}S^c_t-{}^{y}S^c_t>h\}}\left[\phi({}^{x}S^c_t)-\phi({}^{y}S^c_t)\right]\right]\right|
\\&\leq \left|\Ee\left[\mathds{1}_{\{{}^{x}S^c_t-{}^{y}S^c_t\leq h\}}\left[\phi({}^{x}S^c_t)-\phi({}^{y}S^c_t)\right]\right]\right|+2 \|\phi\|_{C[0,T]} \left|\Ee\left[\mathds{1}_{\{{}^{x}S^c_t-{}^{y}S^c_t>h\}}\right]\right|
\\&\leq \varepsilon +2 \|\phi\|_{C[0,T]} \mathbb{P}({}^{x}S^c_t-{}^{y}S^c_t>h).
\end{align*}
 In the last step we use that $\phi\in C_\infty(0, T]$ is uniformly continuous. Note that $h=h_\varepsilon$. 
We have the following two cases: 

1. If $0\leq y\leq x\leq h$, it is clear that $\mathbb{P}({}^{x}S^c_t-{}^{y}S^c_t>h)\to 0$ as $ x-y\to 0$.

2. If $x>h$, it is enough for us to consider the first censoring time \linebreak $\tau_y=\inf\{t\geq 0 : {}^{y}S^c_t\leq 0\}$  for ${}^{y}S^c_t$ and ${}^{x}S^c_{\tau_y}\geq 0$ see Figure 1.   
If $t< \tau_y$ we have $\mathbb{P}({}^{x}S^1_t-{}^{y}S^1_t>h)\to 0$ as $x-y\to 0$. 
\begin{figure}
    \centering    \includegraphics[width=0.9\textwidth]{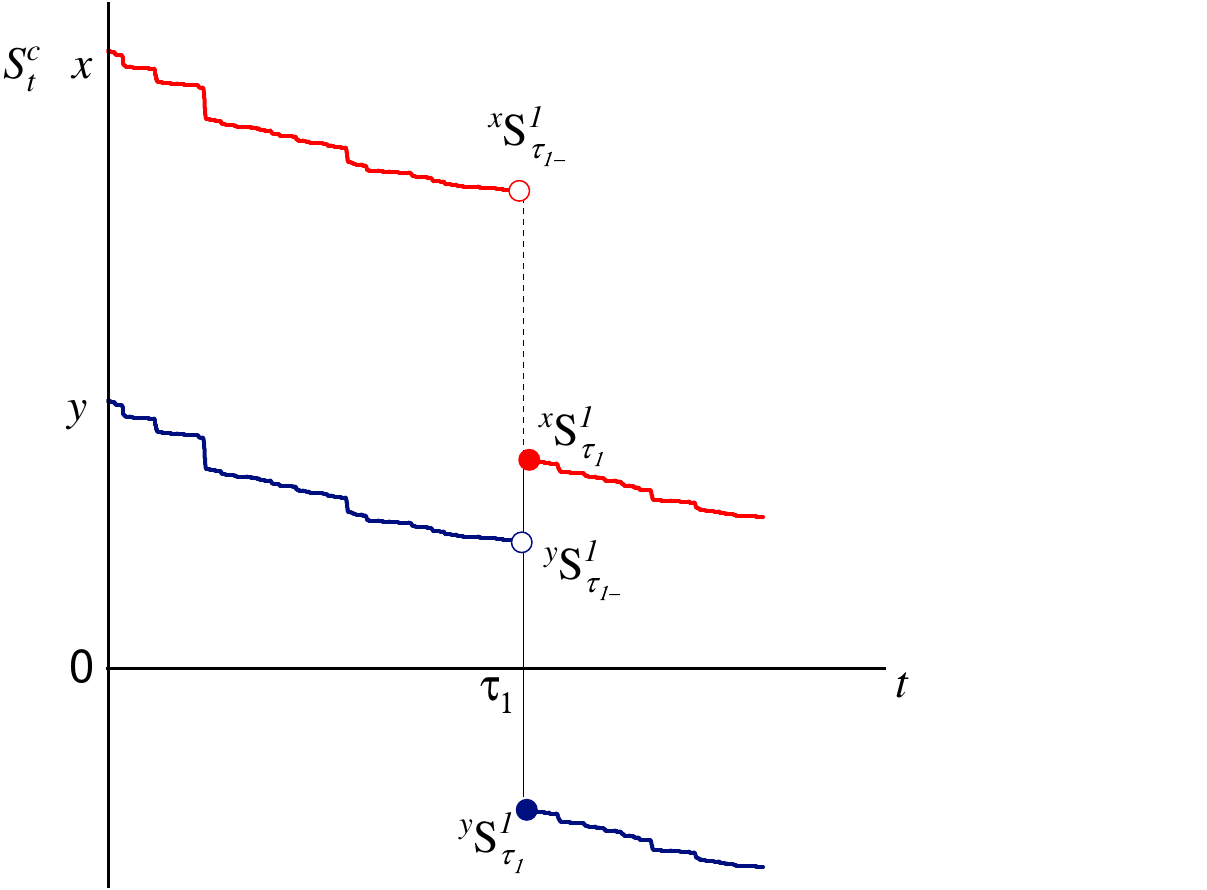}
    \caption{}
\end{figure}
Using the construction of $S^c_\cdot$, we have the following, (see also Figure 4.1):  write $\Delta S_t=S_t-S_{t-}$ for the jump at time $t$, then, 
\begin{align*}
&\mathbb{P}\left[0\in (-\Delta S_{\tau_y}^1+({}^{y}S^1_{\tau_y-}, {}^{x}S^1_{\tau_y})\right]
\\&=\mathbb{P}\left[\Delta S_{\tau_y}^1\in ({}^{y}S^1_{\tau_y-}, {}^{y}S^1_{\tau_y-}+x-y)\right]
\\&=\mathbb{P}\left[{}^{y}S^1_{\tau_y-}+\Delta S_{\tau_y}^1\in (2~{}^{y}S^1_{\tau_y-}, 2~{}^{y}S^1_{\tau_y-}+x-y)\right]
\\&=\mathbb{P}\left[{}^{y}S^1_{\tau_y}\in (2~{}^{y}S^1_{\tau_y-}, 2~{}^{y}S^1_{\tau_y-}+x-y)\right]
\\&=\int_0^{x+y}\int_0^y1_{\{2u \leq v\leq 2u+x-y\}}\mathbb{P}\left[{}^{y}S^1_{\tau_y}\in dv, {}^{y}S^1_{\tau_y-}\in du\right]
\\&=\int_0^{x+y}\int_0^y1_{\{2u \leq v\leq 2u+x-y\}}U(du)\,\bar\mu(dv-u) \xrightarrow[x-y\to 0]{}0.
\end{align*}
where the penultimate equality uses Proposition 2 \cite[p. 76]{1996_Bertoin}.  By Assumption \ref{ass-a} our jump measure is (absolutely) continuous, so the last limits holds.

Finally we show that the generator of $P_t^c \phi(x)=\Ee^x[\phi(S^c_t)]$ is  $-\ced{0}{\cdot}{f}$. Using Theorem \ref{pot-rep-cen-int},  we know the potential of $P_t^c$ is $\cei{0}{\cdot}{f}$  and by \cite[p. 26]{1965_Dynkin} and Theorem \ref{inv-cen-ope-16}, we can get the above results.
\end{proof}

\begin{remark} 
\begin{enumerate}[(1)]
\item $-\ced{0}{\cdot}{f}$ is dissipative.  Thus follows from the structure of $-\ced{0}{\cdot}{f}$: $-\ced{0}{\cdot}{f}$ satisfies  the positive maximum principle and positive maximum principle implies dissipative. 
\item  Using standard theory e.g. \cite[P.28]{1965_Dynkin} Theorem 1.3., the following evolution equation has a unique solution $\phi(t, x)=\Ee^x[g(S_t^c)]$, $g\in \mathcal{D}(-\ced{0}{x}{f})$.
\begin{equation} 
 \left\{
\begin{aligned}
 \partial_t \phi(t,x)&=-\ced{0}{x}{f} \phi(t,x),
\\ \phi(0,x)&=g(x).
\end{aligned}
\right.
\end{equation}
\end{enumerate}
\end{remark}
When solving exit problems by computing the Laplace transform of the lifetime of a killed Markov process, we can obtain the analytical solution to the resolvent equation.
\begin{theorem} Assuming $\lambda<0$, $T>0$ and $g\in C[0, T]$,  we have 
\begin{equation}\label{res-equ}
\Ee^x\left[\int_0^{\tau_\infty}e^{\lambda t}g(S_t^c)\,dt\right]=\sum_{j=0}^\infty \lambda^j (\cei{0}{x}{f})^{j+1} g.
\end{equation}
In particular, $\Ee^x[e^{-\lambda \tau_\infty}]=\sum_{j=0}^\infty\lambda^j (\cei{0}{x}{f})^j \mathds{1}$.
\end{theorem}
\begin{proof}
If $g\in C_\infty(0, T]$, using Hille--Yosida--Ray Theorem 
and Theorem \ref{fell-pro}, the equality holds true as the following equation has a unique solution:
\begin{equation*} 
\left\{
\begin{array}{lll}
\ced{0}{x}{f}\phi(x)&=\lambda \phi(x)+g(x), &x\in (0, T],
\\ \phi(x)&=0,\quad &x=0. 
\end{array}
\right.
\end{equation*}
Now, for any $g\in C[0, T]$, take $g_n\in C_\infty(0, \infty]$ such that $g_n$ converges to $g$ locally uniformly in $(0, T]$. For any fixed $x\in (0, T]$ and $t>0$, we have 
$$\Ee^x[g_n(S_t^c)]\to \Ee^x[g(S_t^c)], \quad \text{as} \quad n\to \infty, $$
by using the dominated convergence theorem, since $\sup_n\|g\|_{C[0, T]}<\infty$. With the dominating function $\sup_n\|g_n\|_{C[0, T]} e^{\lambda t}<\infty$ and  using Lemma \ref{pro-con-pro} $\Ee(\tau_\infty)<\infty$, we apply  again dominated convergence theorem. This gives 
$$\Ee^x\left[\int_0^{\tau_\infty}e^{-\lambda t}g_n(S_t^c)\,dt\right]\to \Ee^x\left[\int_0^{\tau_\infty}e^{-\lambda t}g(S_t^c)\,dt\right], \quad \text{as} \quad n\to \infty.$$
Assume, for a moment that $g\geq 0$. Clearly we can choose  $0\leq g_n\leq g$ such that $g=\sup_n g_n$ (increasing limit). Since $\cei{0}{{\tiny\bullet}}{f}$ is  positivity preserving and linear,  we have $\cei{0}{x}{f} g_n \uparrow \cei{0}{x}{f} g$ as $n$ goes to $\infty$. and so we have 
 $$\sum_{j=0}^\infty\lambda^j (\cei{0}{x}{f})^{j+1} g_n \to \sum_{j=0}^\infty\lambda^j (\cei{0}{x}{f})^{j+1} g, \quad \text{as} \quad n\to \infty.$$
 The general case follows by considering positive and negative parts: $g=g^++g^-$ with $g\in C[0, T]$ and $g_n\to g^+$, $h_n\to g^-$ as $n$ goes to $\infty$. 
Therefore we have proved \eqref{res-equ} for all functions $g\in C[0, T]$.

To prove the special case we take $g=\lambda$, so the left side of \eqref{res-equ} become,
$$\Ee^x\left[\int_0^{\tau_\infty}e^{\lambda t}\lambda\,dt\right]=\Ee^x e^{-\lambda \tau_\infty}-1,$$
 and the right side of \eqref{res-equ} gives 
 $\sum_{j=0}^\infty\lambda^{j+1}(\cei{0}{x}{f})^{j+1}\mathds{1}=\sum_{j=0}^\infty\lambda^j (\cei{0}{x}{f})^j \mathds{1}-1.$
This finishes the proof.
\end{proof}

\appendix
\section{Completely Bernstein functions and Sonine pairs}
     Let $f\in \mathcal{BF}$ be given by \eqref{sub-bf-1} and assume that  $f^\star(\lambda)=\frac{\lambda}{f(\lambda)}$. 
     \begin{equation}\label{sub-bf-1}
 f(\lambda)=\frac{\lambda}{f(\lambda)}=a+b \lambda+\int_0^\infty (1-e^{-\lambda t})\mu (dt)
\end{equation}
     
     Then
\begin{equation}\label{str-2-3-7}
 f^\star(\lambda)=\frac{\lambda}{f(\lambda)}=a^\star+b^\star \lambda+\int_0^\infty (1-e^{-\lambda t})\mu^\star (dt)
\end{equation}
\begin{theorem}\label{cbf-cm} Let  $f\in \mathcal{CBF}$ be a complete Bernstein function with triplet $(a, b, \mu)$, $\bar\mu(x)=\mu[x, \infty)$ and $f^\star=\lambda/f(\lambda)$ with triplet $(a^\star, b^\star, \mu^\star)$, $k(x)=\mu^\star [x, \infty)$ as in \eqref{str-2-3-7}. Then 
$$ 
\frac{f(\lambda)}{\lambda}= \Lscr(a+\bar\mu+b\delta_0;\lambda);
$$
$$
\frac{1}{f(\lambda)}=\Lscr(a^\star+k+b^\star\delta_0;\lambda).
$$
are Stieltjes functions and $a+\bar\mu$ and $a^\star+k$ are completely monotone functions. Moreover, using the convention that $1/\infty=0$. One has

\begin{equation}\label{cbf-3-1-3}
 f^\star (0)=a^\star =\lim_{\lambda \to 0} \frac{\lambda}{f(\lambda)}= \left\{
\begin{aligned}
 & 0,&\quad a>0, \\[-1ex]
\\ & \frac 1{b+\int_0^\infty t \mu(dt)}, &\quad a=0.
\end{aligned}
\right.
\end{equation}
\begin{equation}\label{cbf-3-1-4}
b^\star=\lim_{\lambda\to \infty}\frac{ f^\star (\lambda)}{\lambda}=\lim_{\lambda\to \infty}\frac{1}{f(\lambda)}= \left\{
\begin{aligned}
 & 0,&\quad b>0,\\[0.5ex]
\\& \frac 1{a+\int_0^\infty \mu(dt)}, &\quad b=0.
\end{aligned}
\right.
\end{equation}
and $(a+\bar\mu(x))+b\delta_0$ and $(a^\star+k)+b^\star\delta_0$ is a Sonine pair.  If $b=b^\star=0$ see Table \ref{tab:table1} or \eqref{cbf-3-1-3} and \eqref{cbf-3-1-4}, then $a+\bar\mu(x)$ and $a^\star+k$ is a Sonine pair of completely monotone functions.
\end{theorem}
\begin{proof}
 We have  $f(\lambda)=a+b\lambda +\int_0^\infty (1-e^{-\lambda t})\,\mu(dt)$ and so 
\begin{align*}
\frac{f(\lambda)}{\lambda}&\pomu{Fubini}{=}{}\frac{a}{\lambda}+b+\int_0^\infty \frac{1-e^{-\lambda t}}{\lambda}\,\mu(dt)
\\&\pomu{Fubini}{=}{}\frac{a}{\lambda}+b+\int_0^\infty \int_0^t e^{-\lambda t}\,dt\,\mu(dt)
\\&\omu{Fubini}{=}{} \frac{a}{\lambda}+b+\int_0^\infty \int_x^\infty \,\mu(dt) e^{-\lambda t}\,dt
\\&\pomu{Fubini}{=}{} \frac{a}{\lambda}+b+\int_0^\infty\bar\mu(x) e^{-\lambda t}\,dt
\\&\pomu{Fubini}{=}{} \Lscr(a+b\delta_0+\bar\mu(x); \lambda).
\end{align*}
denote by $f^\star=\frac{\lambda}{f(\lambda)}$ the conjugate Bernstein function. Since $f\in \mathcal{CBF}$, so is $f^\star$  and  using $(a^\star, b^\star,\mu^\star)$ for its triplet, we get 
$$\frac1{f(\lambda)}=\frac{f^\star(\lambda)}{\lambda}=\Lscr(a^\star+b^\star\delta_0+k(x); \lambda).$$
Since  $\frac{1}{\lambda}=\frac{1}{f(\lambda)} \frac{f(\lambda)}{\lambda}$. It is obvious that $a+b\delta_0+\bar\mu(x)$ and $a^\star+b^\star\delta_0+k(x)$ is a Sonine pair (note that $bb^\star=0$, i.e. at least one Sonine factor is a function) and the relation \eqref{cbf-3-1-3} and \eqref{cbf-3-1-4} follows from \cite[Chapter 11]{2012_Schilling} or Theorem 2.3.11 in \cite{2023_Li}.  

Note that (see Table \ref{tab:table1}) $b=b^\star=0$ if and only if we are in the case (1), (2), (4) as shown in Table \ref{tab:table1}. Since $f$ is completely monotone function, $\frac{f(\lambda)}{\lambda}$, $\frac{1}{f(\lambda)}$ are Stieltjes functions, hence $a+\bar\mu$ and $a^\star+k$ are completely monotone functions.
\end{proof}
\begin{table}[h!]
  \begin{center}
    \caption{}
    \label{tab:table1}
    \begin{adjustbox}{width=0.99\textwidth}
\begin{tabular}{c|c|c|c|c|c|c}
Nr. & a &b & $m_0=\int_0^\infty \mu(dt)$ &  $m_1=\int_0^\infty t \,\mu(dt)$&$a^\star$& $b^\star$\\ 
\hline
(1) & 0 &0& -- (no condition)&$\infty$&0&0\\
(2) & 0 &0&$\infty$&$<\infty$&$1/m_1$&0\\
(3) & 0 &0&$<\infty$&$<\infty$&$1/m_1$&$1/m_0$\\
(4) & $>0$ &0&$\infty$&--&0&0\\
(5) & $>0$ &0&$<\infty$&--&0&$1/(a+m_0)$\\
(6) & 0 &$>0$&--&$\infty$&0&0\\
(7) & 0 &$>0$&--&$<\infty$&$1/(b+m_1)$&0\\
(8) & $>0$ &$>0$&--&--&0&0\\
\end{tabular}
\end{adjustbox}
  \end{center}
\end{table} 
\subsection*{Declaration of interests}
{The authors report no conflict of interest.}

\newpage
\bibliographystyle{plain}
\bibliography{Thesis.bib}
\end{document}